\documentclass[11pt,twoside]{article}
\usepackage{fullpage}

\usepackage{amsmath,amssymb,dsfont, amsthm,mathtools,bm}%
\usepackage{framed}
\usepackage{verbatim}
\usepackage{caption,subcaption,graphicx,rotating,multirow}
\usepackage{hyperref,url}
\hypersetup{colorlinks,linkcolor={blue},citecolor={blue}}
\usepackage{enumitem}
\usepackage{natbib}

\newtheorem{question}{Question}
\theoremstyle{plain}
\newtheorem{theorem}{Theorem}

\newtheorem*{remark}{Remark}
\newtheorem{lemma}[theorem]{Lemma}

\newtheorem{corollary}[theorem]{Corollary}
\newtheorem{conjecture}{Conjecture}

\newtheorem{proposition}[theorem]{Proposition}
  \renewcommand{\qed}{\hfill\ensuremath{\blacksquare}}
 \newenvironment{proofof}[1]{{\bf {\bf Proof of #1.}}}{ \hfill \qed}
\newlength{\widebarargwidth}
\newlength{\widebarargheight}
\newlength{\widebarargdepth}
\DeclareRobustCommand{\widebar}[1]{%
  \settowidth{\widebarargwidth}{\ensuremath{#1}}%
  \settoheight{\widebarargheight}{\ensuremath{#1}}%
  \settodepth{\widebarargdepth}{\ensuremath{#1}}%
  \addtolength{\widebarargwidth}{-0.3\widebarargheight}%
  \addtolength{\widebarargwidth}{-0.3\widebarargdepth}%
  \makebox[0pt][l]{\hspace{0.3\widebarargheight}%
    \hspace{0.3\widebarargdepth}%
    \addtolength{\widebarargheight}{0.3ex}%
    \rule[\widebarargheight]{0.95\widebarargwidth}{0.1ex}}%
  {#1}}
  
\makeatletter
\long\def\@makecaption#1#2{
        \vskip 0.8ex
        \setbox\@tempboxa\hbox{\small {\bf #1:} #2}
        \parindent 1.5em  %% How can we use the global value of this???
        \dimen0=\hsize
        \advance\dimen0 by -3em
        \ifdim \wd\@tempboxa >\dimen0
                \hbox to \hsize{
                        \parindent 0em
                        \hfil 
                        \parbox{\dimen0}{\def\baselinestretch{0.96}\small
                                {\bf #1.} #2
                                } 
                        \hfil}
        \else \hbox to \hsize{\hfil \box\@tempboxa \hfil}
        \fi
        }
\makeatother

\newcommand{\half}{\frac{1}{2}}

\newcommand{\normal}{\ensuremath{\mathcal{N}}}

\newcommand{\bigo}{\ensuremath{\mathcal{O}}}
\newcommand\smallo{
  \mathchoice
    {{\scriptstyle\mathcal{O}}}% \displaystyle
    {{\scriptstyle\mathcal{O}}}% \textstyle
    {{\scriptscriptstyle\mathcal{O}}}% \scriptstyle
    {\scalebox{.7}{$\scriptscriptstyle\mathcal{O}$}}%\scriptscriptstyle
  }

\renewcommand{\P}{\operatorname{\mathbb{P}}}
\newcommand{\E}{\operatorname{\mathbb{E}}}

\DeclareMathOperator{\var}{var}

\newcommand{\R}{\mathbb{R}}

\newcommand{\Rp}{\mathbb{R}_{+}}

\newcommand{\indi}{\mathds{1}}

\newcommand{\imnb}{\mathfrak{i}}
\newcommand{\rmd}{\mathrm{d}}

\newcommand{\twonorm}[1]{\left\|#1\right\|_{\ell_2}}

\newcommand{\abs}[1]{\left|#1\right|}

\def\BC{\begin{center}}
\def\EC{\end{center}}
\def\BIT{\begin{itemize}}
\def\EIT{\end{itemize}}
\def\BET{\begin{enumerate}}
\def\EET{\end{enumerate}}
\def\BEQ{\begin{equation}}
\def\EEQ{\end{equation}}

\long\def\comment#1{}

\renewcommand{\intercal}{{\bm \top}}

\newcommand{\RS}{\mathsf{RS}}  
\newcommand{\KL}{\mathsf{KL}}
\newcommand{\TV}{\mathsf{TV}}  
\newcommand{\erfc}{\mathsf{erfc}} 
\newcommand{\err}{\mathsf{err}} 

\newcommand{\utt}{\textup{\texttt{u}}}
\newcommand{\vtt}{\textup{\texttt{v}}}

\newcommand{\y}{\bm{y}}

\newcommand{\w}{\bm{w}}
\renewcommand{\u}{\bm{u}}
\renewcommand{\v}{\bm{v}}

\newcommand{\Y}{\bm{Y}}
\newcommand{\W}{\bm{W}}

\begin{document}
    
\title{\bf{\LARGE{Detection limits in the high-dimensional spiked rectangular model}}}
\author{Ahmed El Alaoui\thanks{Electrical Engineering \& Computer Sciences, UC Berkeley, CA. Email: elalaoui@berkeley.edu}
\and
Michael I. Jordan\thanks{Electrical Engineering \& Computer Sciences, Statistics, UC Berkeley, CA.}
}

\date{}
\maketitle

\vspace*{-.3in} 
\begin{abstract}
We study the problem of detecting the presence of a single unknown spike in a rectangular data matrix, in a high-dimensional regime where the spike has fixed strength and the aspect ratio of the matrix converges to a finite limit. This setup includes Johnstone's spiked covariance model. We analyze the likelihood ratio of the spiked model against an ``all noise" null model of reference, and show it has asymptotically Gaussian fluctuations in a region below---but in general not up to---the so-called BBP threshold from random matrix theory. Our result parallels earlier findings of \cite{onatski2013asymptotic} and \cite{johnstone2015testing} for spherical spikes. We present a probabilistic approach capable of treating generic product priors. In particular, sparsity in the spike is allowed. Our approach is based on Talagrand's interpretation of the cavity method from spin-glass theory. The question of the maximal parameter region where asymptotic normality is expected to hold is left open. This region is shaped by the prior in a non-trivial way. We conjecture that this is the entire paramagnetic phase of an associated spin-glass model, and is defined by the vanishing of the replica-symmetric solution of \cite{lesieur2015mmse}.      
\end{abstract}

%\begin{keywords}
%{\small \noindent\textbf{Keywords:} Spiked random matrix models, high-dimensionality, hypothesis testing, likelihood ratio fluctuations, spin glasses, replica symmetry, the cavity method.} 
%\end{keywords}

\section{Introduction}
The problem of detecting a signal of low-rank structure buried inside a large noise matrix has received enormous attention in the past decade. 
Prominent examples of this problem include the so-called \emph{spiked} or \emph{deformed ensembles} from random matrix theory~\citep{peche2014ICM}. It is particularly interesting to study such problems in the high-dimensional setting where the signal strength is comparable to the noise. This models practical situations in modern data analysis where one wishes to make more complex inferences about a fainter signal as the amount of data accrues.  In this paper we are concerned with the problem of testing the presence of a single weak spike in the data against an ``all-noise" null hypothesis of reference. 

Concretely we consider the observation of an $N \times M$ matrix of the form
\begin{equation}\label{spiked_model}
\Y = \sqrt{\frac{\beta}{N}} \u\v^\intercal + \W, 
\end{equation}
where $\u$ and $\v$ are unknown factors and $\W$ is a matrix with i.i.d.\ noise entries, and we want to test whether $\beta >0$ or $\beta =0$. 
We will assume the noise is standard Gaussian. The parameter $\beta$ represents the strength of the spike, and we assume a high-dimensional setting where $M/N \to \alpha$. The case $\u = \v$ and $\W$ symmetric is referred to as the \emph{spiked Wigner model}.       
When the factors are independent, model~\eqref{spiked_model} can be viewed as a linear model with additive noise and scalar random design:
\[\y_j = \widebar{\beta} v_j \u + \w_j,\] with $1 \le j \le M$, $\widebar{\beta} = \sqrt{\beta/N}$. Assuming $v_j$ has zero mean and unit variance, this is a model of \emph{spiked covariance}: the mean of the empirical covariance matrix $\widehat{\bm \Sigma} = \frac{1}{M}\sum_{j=1}^M \y_j\y_j^\intercal$ is a rank-one perturbation of the identity: $\bm{I}_N + \frac{\beta}{N}\u\u^\intercal$.

The introduction of a particular spiked covariance model by \cite{johnstone2001distribution}---one corresponding to the special case $v_j \sim \normal(0,1)$---has provided the foundations for a rich theory of Principal Component Analysis (PCA), in which the performance of several important tests and estimators is by now well understood~\citep[see, e.g.,][]{ledoit2002some,paul2007asymptotics,nadler2008finite,johnstone2009consistency,amini2008high,berthet2013optimal,dobriban2017sharp}. Parallel developments in random matrix theory have unveiled the existence of sharp transition phenomena in the behavior of the spectrum of the data matrix, where for a spike of strength above a certain \emph{spectral} threshold, the top eigenvalue separates from the remaining eigenvalues which are packed together in a ``bulk" and thus indicates the presence of the spike; below this threshold, the top eigenvalue converges to the edge of the bulk. See \cite{peche2006largest,feral2007largest,capitaine2009largest,benaych2011eigenvalues,benaych2012singular} for results on low-rank deformations of Wigner matrices, and \cite{baik2005phase,baik2006eigenvalues,bai2012sample,bai2008central} for results on spiked covariance models. More recently, an intense research effort has been undertaken to pin down the fundamental limits for both estimating and detecting the spike. 

In a series of papers \citep{korada2009exact,krzakala2016mutual,barbier2016mutual,deshpande2016asymptotic,lelarge2016fundamental_colt,miolane2017fundamental}, the error of the Bayes-optimal estimator has been completely characterized for additive low-rank models with a separable (product) prior on the spike. In particular, these papers confirm an interesting phenomenon discovered by~\cite{lesieur2015mmse,lesieur2015phase}, based on plausible but non-rigorous arguments: for certain priors on the spike, estimation becomes possible---although computationally expensive---below the spectral threshold $\beta=1$. More precisely, the posterior mean overlaps with the spike in regions where the top eigenvector is orthogonal to it. \cite{lesieur2017constrained} provides a full account of these phase transitions in a myriad of interesting situations, the majority of which still await rigorous treatment.  
As for the testing problem, \cite{onatski2013asymptotic,onatski2014signal} and \cite{johnstone2015testing} considered the spiked covariance model for a uniformly distributed unit norm spike, and studied the asymptotics of the likelihood ratio (LR) of a spiked alternative against a spherical null. They showed that the log-LR is asymptotically Gaussian below the spectral threshold  $\alpha\beta^2=1$ \citep[which in this setting is known as the BBP threshold, after][]{baik2005phase}, while it is divergent above it.

However their proof is intrinsically tied to the assumption of a spherical prior. Indeed, by rotational symmetry of the model, the LR depends only on the spectrum, the joint distribution of which is available in closed form. A representation of the LR in terms of a contour integral is then possible (in the single spike case), which can then be analyzed via the method of steepest descent. In a similar but unrelated effort,~\cite{baik2016fluctuations,baik2017fluctuations,baik2017fluctuations_bipartite} studied the fluctuations of the free energy of spherical, symmetric and bipartite versions of the Sherrington--Kirkpatrick (SK) model. This free energy coincides with the log-LR associated with the model~\eqref{spiked_model} for a choice of parameters. The sphericity assumption is again key to their analysis, and both approaches require the execution of very delicate asymptotics and appeal to advanced results from random matrix theory. 
 
In this paper we consider the case of separable priors: we assume that the entries of $\u$ and $\v$ are independent and identically distributed from base priors $P_{\utt}$ and $P_{\vtt}$, respectively, both having bounded support\footnote{Boundedness is required for technical reasons. This rules out the case where one factor is Gaussian.}. We prove fluctuation results for the log-LR in this setting with entirely different methods than used for spherical priors. In particular, our proof is more probabilistic and operates through general principles. The tools we use come from the mathematical theory of spin glasses \citep[see][]{talagrand2011mean1,talagrand2011mean2}. These techniques were successfully used in  \citep{alaoui2017finite} to prove similar results in the spiked Wigner model.

Let us further mention that the region of parameters $(\alpha,\beta)$ we are able to cover with our proof method is optimal when (and only when) $P_{\utt}$  and $P_{\vtt}$ are both symmetric Rademacher. In Section~\ref{sxn:discussion}, we formulate a conjecture on the \emph{maximal} region in which the log-LR has asymptotically Gaussian fluctuations. This region is of course below the BBP threshold, but does \emph{not} extend up to it in general.

 \section{Main results}
Throughout this paper, we assume that the priors $P_{\utt}$ and $P_{\vtt}$ have zero mean, unit variance, and supports bounded in radius by $K_{\utt}$ and $K_{\vtt}$ respectively. Let $\P_{\beta}$ be the probability distribution of the matrix $\Y$ as per~\eqref{spiked_model}. Define $L(\cdot;\beta)$ to be the likelihood ratio, or Radon-Nikodym derivative of $\P_{\beta}$ with respect to $\P_{0}$:
 \[L(\cdot;\beta) \equiv \frac{\rmd \P_{\beta}}{\rmd \P_{0}}.\]
For a fixed $\Y \in \R^{N \times M}$, by conditioning on $\u$ and $\v$, we can write
\[L(\Y;\beta) = \int \exp \Big(\sum_{i,j}\sqrt{\frac{\beta}{N}}Y_{ij}u_iv_j - \frac{\beta}{2N} u_i^2v_j^2\Big)\rmd P_{\utt}^{\otimes N}(\u)\rmd P_{\vtt}^{\otimes M}(\v).\]
 Our main contribution is the following asymptotic distributional result.
\begin{theorem}\label{fluctuations_main_theorem}
Let $\alpha,\beta \ge 0$ such that $K_{\utt}^4K_{\vtt}^4\alpha\beta^2<1$. Then in the limit $N \to \infty$ and $M/N \to \alpha$,
\[\log L(\Y;\beta) \rightsquigarrow \normal\left(\pm\frac{1}{4} \log \left(1-\alpha \beta^2\right),-\frac{1}{2} \log \left(1-\alpha \beta^2\right)\right), \]
where ``$\rightsquigarrow$" denotes convergence in distribution. The sign of the mean is $+$ under the null $\Y \sim \P_{0}$ and $-$ under the alternative $\Y \sim \P_{\beta}$.
\end{theorem}
We mention that fluctuations of this sort were first proved by \cite{aizenman1987some} in a seminal paper in the context of the SK model. 
A consequence of either one of the above statements and Le Cam's first lemma~\citep[Lemma 6.4]{vandervaart2000asymptotic} is the mutual contiguity\footnote{Two sequences of probability measures $(P_n)$ and $(Q_n)$ defined on the same (sequence of) measurable space(s) are said to be mutually contiguous if $P_{n}(A_n) \to 0$ is equivalent to $Q_n(A_n) \to 0$ as $n \to \infty$ for every sequence of measurable sets $(A_n)$.}  between the null and the spiked alternative:

%\vspace{-.1cm}
\begin{corollary}
For $K_{\utt}^4K_{\vtt}^4\alpha\beta^2 <1$, the families of distributions  $\P_{0}$ and $\P_{\beta}$ (indexed by $M,N$) are mutually contiguous in the limit $N \to \infty$, $M/N \to \alpha$.
\end{corollary}
Contiguity implies impossibility of strong detection: there exists no test that, upon observing a random matrix $\Y$ with the promise that it is sampled either from $\P_{0}$ or $\P_{\beta}$, can tell which is the case with asymptotic certainty in this regime. 
We also mention that contiguity can be proved through the second-moment method and its conditional variants, as was done by~\cite{montanari2015limitation,perry2016optimality_annals,banks2017information} for closely related models. However, identifying the right event on which to condition in order to tame the second moment of $L$ is a matter of a case-by-case deliberation. Study of the fluctuations of the log-LR appears to provide a more systematic route: the logarithm has a smoothing effect that kills the wild (but rare) events that otherwise dominate in the second moment. This being said, our result is optimal only in one special case:

When $P_{\utt}$ and $P_{\vtt}$ are symmetric Rademacher, $K_{\utt} = K_{\vtt}=1$, and Theorem~\ref{fluctuations_main_theorem} covers the entire $(\alpha,\beta)$ region where such fluctuations hold. Indeed,  for $\alpha\beta^2>1$, one can distinguish $\P_{\beta}$ from $\P_{0}$ by looking at the top eigenvalue of the empirical covariance matrix $\Y\Y^\intercal$~\citep{benaych2012singular}. So the conclusion of Theorem~\ref{fluctuations_main_theorem} cannot hold in light of the above contiguity argument.  
 Beyond this special case, our result is not expected to be optimal. 
 
\paragraph{Limits of weak detection} Since contiguity implies that testing errors are inevitable, it is natural to aim for tests $T:\R^{N\times M} \mapsto \{0,1\}$ that minimize the sum of the Type-I and Type-II errors:
\[\err(T) = \P_{0}\left(T(\Y) = 1\right)+\P_{\beta}\left(T(\Y) = 0\right).\]
By the Neyman-Pearson lemma, the test minimizing the above error is the likelihood ratio test that rejects the null iff $L(\Y;\beta)>1$. The optimal error is thus 
\[\err_{M,N}^*(\beta) = \P_{0}\left(\log L(\Y;\beta) >0\right)+\P_{\beta}\left(\log L(\Y;\beta) \le 0\right) = 1-D_{\TV}(\P_{\beta},\P_{0}).\]
The symmetry of the means under the null and the alternative in Theorem~\ref{fluctuations_main_theorem} implies that the above Type-I and Type-II errors are equal, and that the total error has a limit:  
\begin{corollary}
For $\alpha,\beta \ge 0$ such that  $K_{\utt}^4K_{\vtt}^4\alpha\beta^2 <1$,
\begin{equation*}\label{optimal_error}
\underset{\substack{ N \to \infty \\ M/N \to \alpha}}{\lim} \err_{M,N}^*(\beta)=1-\underset{\substack{ N \to \infty \\ M/N \to \alpha}}{\lim}D_{\TV}(\P_{\beta},\P_{0}) = \erfc\left(\frac{1}{4}\sqrt{- \log \left(1-\alpha \beta^2\right)}\right), 
 \end{equation*}
where $\erfc(x) = \frac{2}{\sqrt{\pi}}\int_x^\infty e^{-t^2}\rmd t$ is the complementary error function. 
\end{corollary}
Furthermore, our proof of Theorem~\ref{fluctuations_main_theorem} allows us obtain the convergence of the mean (actually, all moments of $\log L$) under $\P_{\beta}$, which corresponds to the Kullback-Liebler divergence of $\P_{\beta}$ to $\P_{0}$:
\begin{proposition}\label{kl_convergence}
For all $\alpha,\beta \ge 0$ such that $K_{\utt}^4K_{\vtt}^4\alpha\beta^2<1$,
\[\underset{\substack{ N \to \infty \\ M/N \to \alpha}}{\lim}D_{\KL}(\P_{\beta},\P_{0}) = -\frac{1}{4} \log \left(1-\alpha \beta^2\right).\]   
\end{proposition}

\section{Replicas, overlaps, Gibbs measures and Nishimori}
A crucial component of the proof involves understanding the convergence properties of certain overlaps between ``replicas." To embark on the argument let us introduce some important notation and terminology. Let $H:\R^{N+M}\to \R$ be the (random) function, which we refer to as a \emph{Hamiltonian}, defined as
\begin{equation} \label{hamiltonian_general}
-H(\u,\v) = \sum_{i,j}\sqrt{\frac{\beta}{N}}Y_{ij}u_iv_j - \frac{\beta}{2N} u_i^2v_j^2,
\end{equation}
where $\Y = (Y_{ij})$ comes from $\P_{\beta}$ or $\P_{0}$. Letting $\rho$ denote the product measure $P_{\utt}^{\otimes N} \otimes P_{\vtt}^{\otimes M}$, we have 
\[L(\Y;\beta) = \int e^{-H(\u,\v)} \rmd \rho(\u,\v).\] 
Let us define the Gibbs average of a function $f: (\R^{N+M})^n\mapsto \R$ of $n$ {replica} pairs $(\u^{(l)},\v^{(l)})_{l=1}^n$ with respect to the Hamiltonian $H$ as
\begin{equation}\label{gibbs_average}
\left \langle f\right\rangle = \frac{\int f \prod_{l=1}^n e^{-H(\u^{(l)},\v^{(l)})}\rmd \rho(\u^{(l)},\v^{(l)})}{\left(\int e^{-H(\u,\v)}\rmd \rho(\u,\v)\right)^n}.
\end{equation} 
This is the mean of $f$ with respect to the posterior distribution of $(\u,\v)$ given $\Y$: $\P_{\beta}(\cdot|\Y)^{\otimes n}$. We interpret the replicas as random and independent draws from this posterior.    
When $\Y \sim \P_{\beta}$ we also allow $f$ to depend on the spike pair $(\u^*,\v^*)$. For two different replicas $(\u^{(l)},\v^{(l)})$ and $(\u^{(l')},\v^{(l')})$ ($l'$ is allowed to take the value $*$) we denote the overlaps of the $\utt$ and $\vtt$ parts, both normalized by $N$, as
\[R^{\utt}_{l,l'} = \frac{1}{N}\sum_{i=1}^N u_i^{(l)} u_i^{(l')} ~~~\mbox{and}~~~ R^{\vtt}_{l,l'} = \frac{1}{N}\sum_{j=1}^M v_j^{(l)} v_j^{(l')}.\]

\subsection{The Nishimori property under $\P_{\beta}$} 
Let's perform the following experiment:
\begin{enumerate}
\item Construct $\u^*\in \R^N$ and $\v^*\in \R^M$ by independently drawing their coordinates from $P_{\utt}$ and $P_{\vtt}$ respectively.
\item Construct $\Y = \sqrt{\frac{\beta}{N}}\u^*\v^{*\intercal} + \W$, where $W_{ij} \sim \normal(0,1)$ are all independent. ($\Y$ is distributed according to $\P_{\beta}$.)
\item Draw $n+1$ independent random vector pairs, $(\u^{(l)},\v^{(l)})_{l=1}^{n+1}$, from $\P_{\beta}((\u,\v) \in \cdot |\Y)$.
\end{enumerate}
By the tower property of expectations, the following equality of joint laws holds  
\begin{equation}\label{nishimori_property}
\left(\Y,(\u^{(1)},\v^{(1)}),\cdots,(\u^{(n+1)},\v^{(n+1)})\right) \overset{\textup{d}}{=}  \left(\Y,(\u^{(1)},\v^{(1)}),\cdots,(\u^{(n)},\v^{(n)}),(\u^{*},\v^{*})\right).
\end{equation}
\citep[See Proposition 15 in][]{lelarge2016fundamental_colt}. This in particular implies that under the alternative $\P_{\beta}$, the overlaps $(R^{\utt}_{1,*},R^{\vtt}_{1,*})$ between replica and spike pairs have the same distribution as the overlaps $(R^{\utt}_{1,2},R^{\vtt}_{1,2})$ between two replica pairs. This is a very important property of the planted (spiked) model, which is usually named after~\cite{nishimori2001statistical} (see Chapter 4). It allows for manipulations that are not possible under the null. For instance, to prove the convergence of the overlap between two replicas, $\E\langle (R^{\utt}_{1,2})^2\rangle  \to 0$, it suffices to prove $\E\langle (R^{\utt}_{1,*})^2\rangle  \to 0$ since the two quantities are equal. The latter turns out to be a much easier task. 

\subsection{Overlap decay implies super-concentration}
Let us now explain how the behavior of the overlaps is related to the fluctuations of $\log L$. For concreteness we consider the null model as an example. Let $\Y\sim \P_{0}$, i.e., $Y_{ij}\sim \normal(0,1)$ all independent. The log-likelihood ratio, seen as a function of $\Y$, is a differentiable function, and
\[\frac{\rmd }{\rmd Y_{ij}} \log L(\Y;\beta) = \sqrt{\frac{\beta}{N}}\langle u_iv_j\rangle.\] 
By the Gaussian Poincar\'e inequality, we can bound the variance by the norm of the gradient as
\[\E\left[(\log L - \E \log L)^2\right] \le \E\left[\twonorm{\nabla \log L}^2\right] = \beta N \E\langle R^{\utt}_{1,2}R^{\vtt}_{1,2}\rangle.\]     
The last equality follows from the fact $\langle u_iv_j\rangle^2 = \langle u_i^{(1)}v_j^{(1)}u_i^{(2)}v_j^{(2)}\rangle$. 
Since our priors have bounded support, we can already bound $R^{\utt}_{1,2}R^{\vtt}_{1,2}$ by $\frac{M}{N}K_{\utt}^2K_{\vtt}^2$, and we deduce that the variance is $\bigo(N)$. In fact, by the Maurey-Pisier inequality~\citep[Theorem 2.2]{Pisier1986}, we can control the moment generating function of $\log L$ by that of $N\langle R^{\utt}_{1,2}R^{\vtt}_{1,2}\rangle$. This implies sub-Gaussian concentration of the former. Observe now that if the quantity $\E\langle R^{\utt}_{1,2}R^{\vtt}_{1,2}\rangle$ decays, then the much stronger result $\var (\log L) = \smallo(N)$ holds. This behavior of unusually small variance is often referred to as ``super-concentration." See~\cite{chatterjee2014superconcentration} for more on this topic. In our case, not only does $\E\langle R^{\utt}_{1,2}R^{\vtt}_{1,2}\rangle$ decay when $\alpha$ and $\beta$ are sufficiently small, but it does so at a rate of $1/N$ so that $N\E\langle R^{\utt}_{1,2}R^{\vtt}_{1,2}\rangle$ converges to a finite limit, and $\var (\log L)$ is constant. This is a first reason why Theorem~\ref{fluctuations_main_theorem} should be expected: if anything, the fluctuations must be of constant order.    

\section{Proof of Theorem~\ref{fluctuations_main_theorem}}
It suffices to prove the fluctuations under one of the hypotheses. Fluctuations under the remaining one comes for free as a consequence of Le Cam's third lemma \citep[or more specifically, the Portmanteau theorem][Theorem 6.6]{vandervaart2000asymptotic}. For the reader's convenience, we present this argument in Appendix~\ref{sxn:fluctuation_equivalence}. We choose to treat the planted case $\Y \sim \P_{\beta}$. The reason is that we are able to achieve control on the overlaps and show their concentration under the alternative in a wider region of parameters $(\alpha,\beta)$ than under the null. This is ultimately due to the Nishimori property~\eqref{nishimori_property}. 
  
We will show the convergence of  the characteristic function of $\log L$  to that of a Gaussian. 
Let $\mu = -\frac{1}{4}\log(1-\alpha \beta^2)$, $\sigma^2=-\frac{1}{2}\log(1-\alpha \beta^2)$,
and let $\phi$ be the characteristic function of the Gaussian distribution $\normal(\mu,\sigma^2)$: for $s \in \R$ and $\imnb^2 = -1$, let $\phi(s) = \exp\{\imnb s \mu -\frac{\sigma^2}{2}s^2\}$.
The following is a more quantitative convergence result that implies Theroem~\ref{fluctuations_main_theorem}. 
\begin{theorem}\label{convergence_of_characteristic_function}
Let $s \in \R$ and $\alpha, \beta \ge 0$. There exists $K=K(s,\alpha,\beta,K_{\utt},K_{\vtt})<\infty$ such that for $M,N$ sufficiently large and $M = \alpha N + \bigo(\sqrt{N})$, the following holds. If $\alpha\beta^2K_{\utt}^4K_{\vtt}^4<1$, then
\[\abs{\E_{\P_{\beta}}\left[e^{\imnb s \log L(\Y;\beta)}\right] - \phi(s)} \le \frac{K}{\sqrt{N}}.\]
\end{theorem}

\begin{remark} 
The condition $M = \alpha N + \bigo(\sqrt{N})$ is assumed only for convenience in order to obtain the rate $1/\sqrt{N}$ in the convergence of the characteristic function. A close inspection of the proof reveals that it can be relaxed to $M/N \to \alpha$ modulo a loss of the convergence rate.      
\end{remark}
 Our approach is to show that the function 
 \[\phi_N(\beta) = \E_{\P_{\beta}}\left[e^{\imnb s \log L(\Y;\beta)}\right]\]
(for $s \in \R$ fixed) is an approximate solution to a differential equation whose solution is the characteristic function of the Gaussian.

\begin{lemma}\label{derivative_phi_planted_lemma} For all $\beta \ge 0$, it holds that
\begin{equation}\label{derivative_phi_planted}
\frac{\rmd}{\rmd \beta} \phi_N(\beta)  = \frac{\imnb s-s^2}{2} N\E\left[\left \langle R^{\utt}_{1,2}R^{\vtt}_{1,2}\right \rangle e^{\imnb s \log L}\right].
\end{equation}
\end{lemma}
\begin{proof}
Since $\Y \sim \P_{\beta}$, we can rewrite the Hamiltonian~\eqref{hamiltonian_general} as
\begin{align*}
-H(\u,\v) &= \sum_{i,j}\sqrt{\frac{\beta}{N}}Y_{ij}u_iv_j - \frac{\beta}{2N} u_i^2v_j^2,\\
&=\sum_{i,j}\sqrt{\frac{\beta}{N}}W_{ij}u_iv_j +\frac{\beta}{N}u_iv_ju_i^*v_j^*- \frac{\beta}{2N} u_i^2v_j^ 2.
\end{align*}
We take a derivative with respect to $\beta$:
\begin{align*}
\frac{\rmd}{\rmd \beta} \phi_N(\beta) &= \imnb s \E\left[\left \langle -\frac{\rmd H}{\rmd \beta}\right \rangle  e^{\imnb s \log L}\right]\\
&=\imnb s \sum_{i,j}\left(\frac{1}{2\sqrt{\beta N}} \E\left[W_{ij}\left \langle u_iv_j \right \rangle  e^{\imnb s \log L}\right] 
-\frac{1}{2 N} \E\left[\left \langle u_i^2v_j^2 \right \rangle  e^{\imnb s \log L}\right]  \right)\\
&~~~+\imnb s \frac{1}{N} \sum_{i,j} \E\left[\left \langle u_iv_ju_i^*v_j^ *\right \rangle  e^{\imnb s \log L}\right].
\end{align*}
The last term is equal to $\imnb s N  \E[ \langle R^{\utt}_{1,*}R^{\vtt}_{1,*} \rangle  e^{\imnb s \log L}]$. As for the first term, since $W_{ij}\overset{\text{ind.}}{\sim} \normal(0,1)$, we use Gaussian integration by parts to obtain 
\begin{align*}
\E\left[W_{ij}\left \langle u_iv_j \right \rangle  e^{\imnb s \log L}\right] &= \E\left[\frac{\rmd}{\rmd W_{ij}}\Big(\left \langle u_iv_j \right \rangle  e^{\imnb s \log L}\Big)\right]\\
&= \sqrt{\frac{\beta}{N}}\left(\E\left[\left \langle u_i^2v_j^2 \right \rangle  e^{\imnb s \log L}\right]
- \E\left[\left \langle u_iv_j \right \rangle^2  e^{\imnb s \log L}\right]+\imnb s \E\left[\left \langle u_iv_j \right \rangle^2  e^{\imnb s \log L}\right]\right).
\end{align*}
Regrouping terms, we get
\begin{align}\label{derivative_phi_prime}
\frac{\rmd}{\rmd \beta} \phi_N(\beta)  &= -\imnb s\frac{N}{2} \E\left[\left \langle R^{\utt}_{1,2}R^{\vtt}_{1,2}\right \rangle e^{\imnb s \log L}\right]
+\imnb s N \E\left[\left \langle R^{\utt}_{1,*}R^{\vtt}_{1,*}\right \rangle e^{\imnb s \log L}\right]\\
&~~~+ (\imnb s)^2 \frac{N}{2} \E\left[\left \langle R^{\utt}_{1,2}R^{\vtt}_{1,2}\right \rangle e^{\imnb s \log L}\right].\nonumber
\end{align}
The first and third terms in~\eqref{derivative_phi_prime} contain overlaps between two replicas while the middle term contains an overlap between one replica and the spike vectors. By the Nishimori property~\eqref{nishimori_property}, we can replace the spike by a second replica in the overlaps appearing in the middle term, and this finishes the proof.
\end{proof} 
\paragraph{A heuristic argument} Let us now heuristically consider what should happen. A rigorous argument will be presented shortly.
If the quantity $N \langle R^{\utt}_{1,2}R^{\vtt}_{1,2} \rangle$ concentrates very strongly about some deterministic value $\theta = \theta(\alpha,\beta)$, we would expect that the Gibbs averages in~\eqref{derivative_phi_planted} would behave approximately independently from $\log L$, and we would obtain the following differential equation
\[\frac{\rmd}{\rmd \beta} \phi_N(\beta) \simeq \half\left(\imnb s - s^2\right)\theta\phi_N(\beta).\]
Since $\phi_N(0)=1$, one obtains  $\phi_N(\beta) \simeq \exp\{\half(\imnb s - s^2) \int_0^\beta \theta \rmd\beta'\}$ by integrating over $\beta$, and the result would follow. The concentration assumption we used is commonly referred to as \emph{replica-symmetry} or \emph{the replica-symmetric ansatz} in the statistical physics literature. Most of the difficulty of the proof lies in showing rigorously that replica symmetry indeed holds. 

\paragraph{Sign symmetry between $\P_{\beta}$ and $\P_{0}$} One can execute the same argument under the null model. Since there is no planted term in the Hamiltonian, the analogue of~\eqref{derivative_phi_prime} one obtains does not contain the middle term. Hence the differential equation one obtains is
\[\frac{\rmd}{\rmd \beta} \phi_N(\beta) \simeq \half\left(-\imnb s - s^2\right)\theta\phi_N(\beta).\]  
This is one way to interpret the sign symmetry of the means of the limiting Gaussians under the null and the alternative: the interaction of one replica with the planted spike under the planted model accounts for twice the contribution of the interaction between two independent replicas, and this flips the sign of the mean. 

We now replace the above heuristic with a rigorous statement. Recall that $\Y \sim \P_{\beta}$.
\begin{proposition}\label{asymptotic_decoupling}
For $s \in \R$ and $\alpha,\beta \ge 0$ such that $\alpha\beta^2K_{\utt}^4K_{\vtt}^4<1$, there exist a constant $K = K(s,\alpha,\beta,K_{\utt},K_{\vtt}) <\infty$ such that 
\[N\E\left[\left\langle R^{\utt}_{1,2}R^{\vtt}_{1,2}\right\rangle e^{\imnb s \log L}\right] = \frac{\alpha \beta}{1-\alpha\beta^2}\E\left[e^{\imnb s \log L}\right]+ \delta,\] 
where $|\delta| \le K/\sqrt{N}$.
Moreover, $K$, seen as a function of $\beta$, is bounded on any interval $[0,\beta']$ when $\alpha\beta'^2K_{\utt}^4K_{\vtt}^4<1$. 
\end{proposition}

Taking $s=0$, we see that $\theta =  \frac{\alpha \beta}{1-\alpha\beta^2}$. Proposition~\ref{asymptotic_decoupling} vindicates replica symmetry, and its proof occupies the majority of the rest of the manuscript. 

\vspace{.2cm}  
\noindent\begin{proofof}{Theorem~\ref{convergence_of_characteristic_function}}
Plugging the results of Proposition~\ref{asymptotic_decoupling} in the derivative computed in Lemma~\ref{derivative_phi_planted_lemma}, we obtain
\[\frac{\rmd}{\rmd \beta} \phi_N(\beta) = \left(\frac{\imnb s - s^2}{2}\frac{\alpha \beta}{1-\alpha\beta^2}\right)\phi_N(\beta) + \delta,\]
where $|\delta| \le \frac{K}{\sqrt{N}}\max\{|s|,s^2\}$, and $K$ is the constant from Proposition~\ref{asymptotic_decoupling}. 
Integrating w.r.t.\ $\beta$ we obtain
\[\abs{\phi_N(\beta) - \phi(s)} \le \frac{K'}{\sqrt{N}},\]
where $K'$ depends on $\alpha,\beta,s$ and $K_{\utt},K_{\vtt}$, and $K' <\infty$ as long as $\alpha\beta^2 K_{\utt}^4K_{\vtt}^4<1$. 
\end{proofof}

\vspace{.2cm}
Let us prove in passing the convergence of the $\KL$ divergence between the null and alternative. 

\vspace{.1cm}
\noindent\begin{proofof}{Proposition~\ref{kl_convergence}}
Similarly to the computation of the derivative of $\phi_N$, we can obtain
\[\frac{\rmd}{\rmd \beta} \E_{\P_{\beta}}\log L(\Y;\beta) =  -\frac{N}{2} \E\left \langle R^{\utt}_{1,2}R^{\vtt}_{1,2}\right \rangle
+N \E\left \langle R^{\utt}_{1,*}R^{\vtt}_{1,*}\right \rangle = \frac{N}{2} \E\left \langle R^{\utt}_{1,2}R^{\vtt}_{1,2}\right \rangle,\]
where we used the Nishimori property. By Proposition~\ref{asymptotic_decoupling} with $s=0$, this derivative is $K/\sqrt{N}$ away from $\half\frac{\alpha \beta}{1-\alpha\beta^2}$. Integration and boundedness of $K$ finishes the proof. 
\end{proofof}

\section{Overlap convergence}
The question of overlap convergence is purely a spin glass problem. We will use the machinery developed by Talagrand to solve it. In particular, a crucial use is made of the cavity method and Guerra's interpolation scheme. In this section, we present the main underlying ideas. The arguments are technically involved (but conceptually simple) so we delay their full execution to the Appendix. We refer to~\cite{talagrand2007obnoxious} for a leisurely high-level introduction to these ideas. 

\subsection{Sketch of proof of Proposition~\ref{asymptotic_decoupling}}
The basic idea is to show that the quantities of interest approximately obey a self-consistent (or self-bounding) property, the error terms of which can be controlled. This approach will be used at different stages of the proof.  
We will show that 
\[N\E\left[\left\langle R^{\utt}_{1,2}R^{\vtt}_{1,2}\right\rangle e^{\imnb s \log L}\right] = \alpha\beta \E\left[e^{\imnb s \log L}\right]+\alpha\beta^2 N \E\left[\left\langle R^{\utt}_{1,2}R^{\vtt}_{1,2}\right\rangle e^{\imnb s \log L}\right] + \delta, \]
where $\delta$ is the error term. This will be achieved in two steps. We first prove
\begin{equation}\label{cavity_uv_to_vv}
N\E\left[\left\langle R^{\utt}_{1,2} R^{\vtt}_{1,2}\right\rangle e^{\imnb s \log L}\right] = N\beta \E\left[\left\langle (R^{\vtt}_{1,2})^2\right\rangle e^{\imnb s \log L}\right] + \delta,
\end{equation}
via a cavity on $N$, i.e., by isolating the effect of the last variable $u_N$ on the rest of the variables. We then show 
\begin{equation}\label{cavity_vv_to_uv}
N\E\left[\left\langle (R^{\vtt}_{1,2})^2\right\rangle e^{\imnb s \log L}\right]  = \frac{M}{N}  \E\left[e^{\imnb s \log L}\right]  + M \beta \E\left[\left\langle R^{\utt}_{1,2} R^{\vtt}_{1,2}\right\rangle e^{\imnb s \log L}\right] +\delta,
\end{equation}
via a cavity on $M$, i.e., isolating the effect of $v_M$. In the arguments leading to~\eqref{cavity_uv_to_vv} and~\eqref{cavity_vv_to_uv}, we accumulate error terms that are proportional to the third moments of the overlaps:
\begin{equation}\label{error_term_delta}
\delta \lesssim N\E\left\langle |R^{\utt}_{1,2}|^{3}\right\rangle + N\E\left\langle |R^{\vtt}_{1,2}|^{3}\right\rangle,
\end{equation}
where we hide constants depending on $\alpha$ and $\beta$.  
These cavity equations impose only a mild restriction on the parameters so that our bounds go in the right direction, namely that $\alpha \beta^2 <1$. This is about to change. We prove that $\delta = \bigo(1/\sqrt{N})$ with methods that impose the stronger restrictions on $(\alpha,\beta)$ that ultimately appear in the final result.

\subsection{Convergence in the planted model: from crude estimates to optimal rates}
We prove overlap convergence under the alternative. Let $\Y \sim \P_{\beta}$. 
\begin{proposition}\label{fourth_moment_planted}
For all $\alpha,\beta \ge 0$ such that $K_{\utt}^4K_{\vtt}^4\alpha\beta^2 < 1$, there exists $K = K(\alpha,\beta) <\infty$ such that
\[\E\left\langle (R^{\utt}_{1,2})^4\right\rangle \vee \E\left\langle (R^{\vtt}_{1,2})^4\right\rangle \le \frac{K}{N^2}.\]
\end{proposition}
The proof proceeds as follows.
We use the cavity method to show the following self-consistency equations:
\begin{align}
\E\left \langle (R^{\utt}_{1,2})^4\right \rangle &= \alpha \beta^2\E\left \langle (R^{\utt}_{1,2})^4\right \rangle + \widebar{M}_{\utt}+\delta_{\utt} \label{bound_fourth_moment_u},\\
\E\left \langle (R^{\vtt}_{1,2})^4\right \rangle &= \alpha \beta^2\E\left \langle (R^{\vtt}_{1,2})^4\right \rangle + \widebar{M}_{\vtt}+\delta_{\vtt},
\label{bound_fourth_moment_v}
\end{align} 
where $|\widebar{M}_{\utt}|, |\widebar{M}_{\vtt}|$ are bounded by sums of expectations of monomials of degree five in the overlaps $R^{\utt}$ and $R^{\vtt}$:
\[|\widebar{M}_{\utt}| \lesssim \sum_{a,b,c,d} \E\left \langle \abs{(R^{\utt}_{1,2})^3R^{\utt}_{a,b}R^{\utt}_{c,d}}\right \rangle + \E\left \langle \abs{(R^{\utt}_{1,2})^3R^{\vtt}_{a,b}R^{\vtt}_{c,d}}\right \rangle,\] 
\[ |\widebar{M}_{\vtt}| \lesssim \sum_{a,b,c,d} \E\left \langle \abs{(R^{\vtt}_{1,2})^3R^{\vtt}_{a,b}R^{\vtt}_{c,d}}\right \rangle + \E\left \langle \abs{(R^{\vtt}_{1,2})^3R^{\utt}_{a,b}R^{\utt}_{c,d}}\right \rangle,\] 
where the sum is over a finite number of combinations $(a,b,c,d)$, and 
\begin{equation*}
\delta_{\utt} \lesssim \frac{1}{N}\E\left \langle (R^{\utt}_{1,2})^2\right \rangle + \bigo\Big(\frac{1}{N^2}\Big),
\qquad 
\delta_{\vtt} \lesssim \frac{1}{N}\E\left \langle (R^{\vtt}_{1,2})^2\right \rangle + \bigo\Big(\frac{1}{N^2}\Big).
\end{equation*}
These results hold for \emph{all} $\alpha, \beta \ge 0$. From here, further progress is unlikely unless one has \emph{a priori} knowledge that the overlaps are unlikely to be large, so that the fifth-order terms do not overwhelm the main terms. More precisely, suppose that we are able to prove the following crude bound on the overlaps: for $\epsilon>0$, there is $K = K(\epsilon,\alpha,\beta)>0$ such that
\begin{equation}\label{crude_bound}
\E\left \langle \indi\left\{\big|R^{\utt}_{1,2}\big| \ge \epsilon \right\}\right \rangle \vee \E\left \langle \indi\left\{\big|R^{\vtt}_{1,2}\big| \ge \epsilon \right\}\right \rangle \le Ke^{-N/K}.
\end{equation}
Then the fifth-order terms can be controlled by fourth-order terms as follows:
\begin{align*}
\E\left \langle \abs{(R^{\utt}_{1,2})^3R^{\vtt}_{a,b}R^{\vtt}_{c,d}}\right \rangle &\le \epsilon \E\left \langle \abs{(R^{\utt}_{1,2})^3R^{\vtt}_{a,b}}\right \rangle + K_{\utt}^6 K_{\vtt}^4 K e^{-N/K}\\
&\le\epsilon M + Ke^{-N/K},
\end{align*}
where $M = \E \langle (R^{\utt}_{1,2})^4 \rangle \vee \E \langle (R^{\vtt}_{1,2})^4 \rangle$, and the last step is by H\"{o}lder's inequality. This way, $\widebar{M}_{\utt}$ and $\widebar{M}_{\vtt}$ are controlled. Now it remains to control $\delta_{\utt}$ and $\delta_{\vtt}$. We could re-execute the cavity argument on the second moment instead of the fourth, and this would allow us to obtain $\E \langle (R^{\utt}_{1,2})^2 \rangle \vee \E \langle (R^{\vtt}_{1,2})^2 \rangle \le K/N$. We instead use a shorter argument based on an elegant \emph{quadratic replica coupling} technique of \cite{guerra2002quadratic} to prove this. This is presented in Appendix~\ref{sxn:proof_of_convergence_second_moment}.  Plugging these estimates into~\eqref{bound_fourth_moment_u} and~\eqref{bound_fourth_moment_v}, we obtain  
 \begin{align*}
\E\left \langle (R^{\utt}_{1,2})^4\right \rangle &\le \alpha \beta^2\E\left \langle (R^{\utt}_{1,2})^4\right \rangle + K\epsilon M + \delta',\\
\E\left \langle (R^{\vtt}_{1,2})^4\right \rangle &\le \alpha \beta^2\E\left \langle (R^{\vtt}_{1,2})^4\right \rangle + K\epsilon M + \delta',
\end{align*} 
where $\delta' \le K/N^2 + K e^{-N/K}$, and this implies the desired result for $\epsilon$ sufficiently small. 

The a priori bound~\eqref{crude_bound} is proved via an interpolation argument at fixed overlap, combined with concentration of measure, and is presented in Appendices~\ref{sxn:interpolation_at_fixed_overlap} and~\ref{sxn:proof_of_crude_bound}. These arguments impose a restriction on the parameters $(\alpha,\beta)$ that shows up in the final result. Finally, Proposition~\ref{fourth_moment_planted} allow us to conclude (via Jensen's inequality) that the error term $\delta$ displayed in~\eqref{error_term_delta} is bounded by $K/\sqrt{N}$.

\section{Discussion}
\label{sxn:discussion}
The limiting factor in our approach to prove LR fluctuations is the need for precise non-asymptotic control of moments of the overlaps $R_{1,2}^{\utt}$ and $R_{1,2}^{\vtt}$ under the expected Gibbs measure $\E \langle \cdot \rangle$. We were able to reach this level of control only in a restricted regime. This is  due to the failure of our approach to prove the crude estimate~\eqref{crude_bound} in a larger region. In this section, we formulate a conjecture on the largest region where these fluctuations and overlap decay should occur. In one sentence, this should be the entire \emph{annealed} or \emph{paramagnetic} region of the model, as dictated by the vanishing of its replica-symmetric ($\RS$) formula. We shall now be more precise.        

Let $z \sim \normal(0,1)$, $u^* \sim P_\utt$ and $v^* \sim P_\vtt$ all independent. Define 
\begin{align*}
\psi_{\utt}(r) &:= \E_{u^*,z} \log \int \exp\left(\sqrt{r}zu + r uu^* - \frac{r}{2} u^2\right) \rmd P_\utt(u),\\
\psi_{\vtt}(r) &:= \E_{v^*,z} \log \int \exp\left(\sqrt{r}zv + r vv^* - \frac{r}{2} v^2\right) \rmd P_\vtt(v).
\end{align*}

Moreover, define the $\RS$ potential as
\[F(\alpha,\beta, q_{\utt}, q_{\vtt}) := \psi_{\utt}(\beta q_{\vtt}) + \alpha \psi_{\vtt}(\beta  q_{\utt}) -\frac{\beta  q_{\utt} q_{\vtt}}{2}.\]
and finally define the $\RS$ formula as
\[\phi_{\RS}(\alpha,\beta) := \sup_{q_{\vtt} \ge 0}~ \inf_{q_{\utt}\ge 0}~ F(\alpha,\beta, q_{\utt}, q_{\vtt}).\]
It was argued by~\cite{lesieur2015mmse} based on the plausibility of the replica-symmetric ansatz, and then proved by~\cite{miolane2017fundamental}, that in the limit $N \to \infty, M/N \to \alpha$,\\ $\frac{1}{N}\E_{\P_{\beta}}\log L(\Y;\beta) \rightarrow \phi_{\RS}(\alpha,\beta)$ for all $\alpha,\beta \ge 0$. \citep[See also][for results in a more general setup.]{barbier2017phase} Of course, by change of measure and Jensen's inequality, 
\[\E_{\P_{\beta}} \log  L(\Y;\beta) = \E_{\P_{0}} L(\Y;\beta)\log  L(\Y;\beta) \ge 0,\]
for all $M,N$; therefore $\phi_{\RS}$ is always nonnegative. Let 
\[\Gamma = \left\{(\alpha,\beta) \in \Rp ~:~ \phi_{\RS}(\alpha,\beta) = 0\right\}.\]

It is not hard to prove the following lemma by analyzing the stability of $(0,0)$ as a stationary point of the $\RS$ potential:   
\begin{lemma}
$\Gamma \subseteq \{(\alpha,\beta) \in \Rp~:~\alpha \beta^2 \le 1\}$.
\end{lemma}
This lemma tells us (unsurprisingly) that $\Gamma$ is entirely below the BBP threshold. The inclusion may or may not be strict depending on the priors $P_{\utt}$ and $P_{\vtt}$. For instance, there is equality of the above sets if $P_{\utt}$ and $P_{\vtt}$ are symmetric Rademacher and/or Gaussian respectively.     
One case of strict inclusion is when $P_{\vtt}$ is Gaussian $\normal(0,1)$ and $P_{\utt}$ is a sparse Rademacher prior, $\frac{\rho}{2}\delta_{1/\sqrt{\rho}}+ (1-\rho)\delta_{0} + \frac{\rho}{2}\delta_{-1/\sqrt{\rho}}$, for sufficiently small $\rho$ (e.g., $\rho = .04$). This is a canonical model for sparse principal component analysis. In this case, there is a region of parameters below the BBP threshold where the posterior mean $\E[\u^*|\Y]$ ($= \langle \u\rangle$ in our notation) has a non-trivial overlap with the spike $\u^*$, while the top eigenvector of the empirical covariance matrix $\Y\Y^\intercal$ is orthogonal to it. Estimation becomes impossible only in the region $\Gamma$, so the following conjecture is highly plausible:

\begin{conjecture}
Let $\Gamma'$ be the interior of $\Gamma$. For all $(\alpha,\beta) \in \Gamma'$, 
\[\log L(\Y,\beta) \rightsquigarrow \normal\left(\pm \frac{1}{4}\log(1-\alpha\beta^2),-\frac{1}{2}\log(1-\alpha\beta^2)\right),\]
where the plus sign holds under the null $\P_{0}$ and the minus sign under the alternative $\P_{\beta}$.
\end{conjecture}
 
Our conjecture is formulated only in the interior of $\Gamma$; this is not a superfluous condition since diverging behavior may appear at the boundary. Moreover, this conjecture is about the \emph{maximal} region in which such fluctuations can take place. This is not difficult to show. By (sub-Gaussian) concentration of the normalized likelihood ratio, we have for $\epsilon >0$
\[\P_{\beta}\Big(\frac{1}{N}\log L(\Y;\beta) - \phi_{\RS}(\alpha,\beta) \le -\epsilon\Big)  \longrightarrow 0,\]  
where $K = K(\alpha,\beta)<\infty$. This already shows that $\log L$ must grow with $N$ under the alternative if $\phi_{\RS} >0$. As for the behavior under the null,  
the same sub-Gaussian concentration holds, although the expectation is not known (see Question~\ref{our_question}):
\[\P_{0}\Big(\frac{1}{N}\log L(\Y;\beta) - \frac{1}{N}\E_{\P_{0}}\log L(\Y;\beta) \ge \epsilon\Big) \longrightarrow 0.\]
We do know however that the above expectation is non-positive, by Jensen's inequality. Therefore if $(\alpha,\beta)$ are such that $\phi_{\RS} >0$, one can distinguish $\P_{\beta}$ from $\P_{0}$ with asymptotic certainty by testing whether $\frac{1}{N}\log L(\Y;\beta)$ is above or below (say) $\half\phi_{\RS}(\alpha,\beta)$. This implies that $\P_{\beta}$ and $\P_{0}$ are not contiguous outside $\Gamma$. This---short of proving that $\log L$ grows in the negative direction with $N$---shows that the fluctuations cannot be of the above form under the null, since this would contradict Le Cam's first lemma.

The difficulty we encountered in our attempts to prove the above conjecture is a loss of control over the overlaps $R_{1,2}^{\utt}$ and $R_{1,2}^{\vtt}$ near the boundary of the set $\Gamma$. The interpolation bound at fixed overlap (between a replica and the spike) we used under the alternative $\P_{\beta}$ is vacuous beyond the region $\alpha\beta^2 <(K_{\utt}K_{\vtt})^{-4}$. It is possible that the latter bound could be marginally improved by more careful analysis, but this is unlikely to yield the optimal result since no information about $\phi_{\RS}$ is used in the proof. One can imagine refining this technique by constraining two replicas and using an interpolation with broken replica-symmetry, in the spirit of the ``2D" Guerra-Talagrand bound~\citep{guerra2003broken,talagrand2011mean2}. Although this strategy is successful in the symmetric model where $\u = \v$ it is not at all obvious why such an interpolation bound should be true in the bipartite case: in the analysis, certain terms that are hard to control have a sign in the symmetric case, hence they can be dropped to obtain a bound. This is no longer true (or at least not obviously so) in the bipartite case.

 Another interesting question concerns the LR asymptotics under the null, outside $\Gamma$. While under the alternative $\P_{\beta}$, the normalized log-likelihood ratio converges to the $\RS$ formula $\phi_{\RS}$ for all $(\alpha,\beta)$, no such simple formula is expected to hold under the null. Even the existence of a limit seems to be unknown. 
\begin{question}\label{our_question}
Does $\frac{1}{N}\E_{\P_{0}}\log L(\Y;\beta)$ have a limit for all $(\alpha,\beta)$? If so, what is its value? 
\end{question}
We refer to \cite{barra2011equilibrium,barra2014mean} and \cite{auffinger2014free} for some progress on the replica-symmetric phase, and \cite{panchenko2015multispecies} for progress on the related problem of the ``multispecies" SK model at all temperatures.

\begin{small}

\bibliographystyle{apalike}
\bibliography{../phase_transitions}
\end{small}

%\newpage
\appendix

\section{Fluctuation equivalence}
\label{sxn:fluctuation_equivalence}
We explain in this appendix how the fluctuation result under $\P_{\beta}$ implies the corresponding fluctuation result under $\P_{0}$. This is a consequence of the Portmanteau characterization of convergence in distribution. The argument can be made in the other direction as well. Assume that 
\[\log L(\Y;\beta) \rightsquigarrow \normal(\mu,\sigma^2),\]
 for $\Y \sim \P_{\beta}$, where $\mu = \half\sigma^2$.   
By the Portmanteau theorem~\citep[Lemma 2.2]{vandervaart2000asymptotic}, this is equivalent to the assertion
\begin{equation}\label{portmanteau_1}
 \liminf \E_{\P_{\beta}}\left[f(\log L)\right] \ge \E\left[f(Z)\right],
 \end{equation}
where $Z \sim \normal(\mu,\sigma^2)$ for all nonnegative continuous functions $f:\R \mapsto \Rp$. On the other hand, by a change of measure (and absolute continuity of $\P_{0}$ w.r.t $\P_{\beta}$), we have that for such an $f$,
\[\E_{\P_{0}}\left[f(\log L)\right] = \E_{\P_{\beta}}\left[ \frac{\rmd \P_{0}}{\rmd \P_{\beta}}  f(\log L)\right] = \E_{\P_{\beta}}\left[e^{-\log L}  f(\log L)\right].\]
The function $g:x\mapsto e^{-x}f(x)$ is still nonnegative continuous, so by~\eqref{portmanteau_1}, we have
\begin{equation}\label{portmanteau_2}
\liminf \E_{\P_{0}}\left[f(\log L)\right] \ge \E\left[e^{-Z}f(Z)\right].
\end{equation}
Since $\mu = \half\sigma^2$, 
\[\E\left[e^{-Z}f(Z)\right] = \int f(x)e^{-x}  e^{-(x-\mu)^2/2\sigma^2}\frac{\rmd x}{\sqrt{2 \pi \sigma^2}} = \int f(x) e^{-(x+\mu)^2/2\sigma^2}\frac{\rmd x}{\sqrt{2 \pi \sigma^2}} = \E\left[f(Z')\right],\]  
where $Z' \sim \normal(-\mu,\sigma^2)$. Since~\eqref{portmanteau_2} is valid for every nonnegative continuous $f$, the result 
\[\log L(\Y;\beta) \rightsquigarrow \normal(-\mu,\sigma^2)\]
under $\P_{0}$ follows. 

\section{Notation and useful lemmas}
\label{sxn:preliminaries}
We make repeated use of interpolation arguments in our proofs. In this section, we state a few elementary lemmas we subsequently invoke several times. We denote the overlaps between replicas when the last variables are deleted by a superscript $``-"$ :
\[R_{l,l'}^{\utt -} = \frac{1}{N}\sum_{i=1}^{N-1}u^{(l)}_iu^{(l')}_i ~~~\mbox{and}~~~ R_{l,l'}^{\vtt -} = \frac{1}{N}\sum_{j=1}^{M-1}v^{(l)}_jv^{(l')}_j.\]
If $\{H_t: t \in [0,1]\}$ is a generic family of random Hamiltonians, we let $\langle \cdot \rangle_t$ be the corresponding Gibbs average, and $\nu_t(f) =\E\left \langle f\right \rangle_t$, where the expectation is over the randomness of $H_t$. We will often write $\nu$ for $\nu_1$.

In our executions of the cavity method, we use interpolations that isolate one last  variable (either $u_N$ or $v_M$) from the rest of the system. Taking the first case an example, we consider 
\begin{align*}\label{interpolation_u}
-H_t(\u,\v) &=  \sum_{i=1}^{N-1} \sum_{j=1}^{M} \sqrt{\frac{\beta}{N}}W_{ij}u_i v_j + \frac{\beta}{N}u_iu_i^*v_jv_j^*- \frac{\beta}{2N}u_i^2v_j^2 \\
&~~+  \sum_{j=1}^M \sqrt{\frac{\beta t}{N}}W_{Nj}u_N v_j + \frac{\beta t}{N}u_Nu_N^*v_jv_j^* - \frac{\beta t}{2N}u_N^2v_j^2. \nonumber
\end{align*}  
\begin{lemma}\label{derivative_gibbs_average}
Let $f$ be a function of $n$ replicas $(\u^{(l)},\v^{(l)})_{1\le l \le n}$. Then 
\begin{align*}
\frac{\rmd}{\rmd t}\nu_t(f) &= \frac{\beta}{2} \sum_{1\le l\neq l' \le n} \nu_t(R^{\vtt}_{l,l'}u^{(l)}u^{(l')}f) 
- \frac{\beta}{2} n \sum_{l=1}^n \nu_t(R^{\vtt}_{l,n+1} u^{(l)}u^{(n+1)}f) \\
&~~+ \beta n \sum_{l=1}^n \nu_t(R^{\vtt}_{l,*} u^{(l)}u^{*}f) 
- \beta n  \nu_t(R^{\vtt}_{n+1,*} u^{(n+1)}u^{*}f) \\
&~~+  \beta \frac{n(n+1)}{2} \nu_t(R^{\vtt}_{n+1,n+2} u^{(n+1)}u^{(n+2)}f).
\end{align*}
\end{lemma}
\begin{proof}
This is a simple computation based on Gaussian integration by parts, similarly to Lemma~\ref{derivative_phi_planted}. 
\end{proof}
The next lemma allows us to control interpolated averages by averages at time $1$. 
\begin{lemma}\label{bound_time_dependent_average}
Let $f$ be a nonnegative function of $n$ replicas $(\u^{(l)},\v^{(l)})_{1\le l \le n}$. Then for all $t \in [0,1]$
\[\nu_t(f) \le K(n,\alpha,\beta) \nu(f).\]
\end{lemma}
\begin{proof}
This is a consequence of Lemma~\ref{derivative_gibbs_average}, boundedness of the variables $u_i$ and $v_j$, and Gr\"onwall's lemma. 
\end{proof}
It is clear that Lemma~\ref{bound_time_dependent_average} also holds if we switch the roles of $\u$ and $\v$ and extract $v_M$ instead (so that $\nu_t$ is defined accordingly).  

\section{Proof of Proposition~\ref{asymptotic_decoupling}}
\label{sxn:proof_of_asymptotic_decoupling}
We make use of two interpolation arguments; the first one extracts the last variable $u_N$ from the system, and the second one extracts $v_M$. This allows to establish the self-consistency equations~\eqref{cavity_uv_to_vv} and~\eqref{cavity_vv_to_uv}.  We will assume decay of the forth moments of the overlaps, i.e., we assume Proposition~\ref{fourth_moment_planted} (which we prove in Appendix~\ref{sxn:proof_of_fourth_moment_planted}), and this allows us the prove that the error terms emerging from the cavity method converge to zero.
Recall that the Nishimori property implies 
 \[\E\left[\left\langle R^{\utt}_{1,2}R^{\vtt}_{1,2}\right\rangle e^{\imnb s \log L}\right] =\E\left[\left\langle R^{\utt}_{1,*}R^{\vtt}_{1,*}\right\rangle e^{\imnb s \log L}\right].\]
As it turns out, it is more convenient to work with the right-hand side.

\subsection{Cavity on $N$} 
By symmetry of the $\utt$ variables, we have
\[\E\left[\left\langle R^{\utt}_{1,*}R^{\vtt}_{1,*}\right\rangle e^{\imnb s \log L}\right] = \E\left[\left\langle u_N^{(1)}u_N^{*} R^{\vtt}_{1,*}\right\rangle e^{\imnb s \log L}\right].\]
Now we consider the interpolating Hamiltonian
\begin{align*}
-H_t(\u,\v) &=  \sum_{i=1}^{N-1} \sum_{j=1}^{M} \sqrt{\frac{\beta}{N}}W_{ij}u_i v_j + \frac{\beta}{N}u_iu_i^*v_jv_j^*- \frac{\beta}{2N}u_i^2v_j^2 \\
&~~+  \sum_{j=1}^M \sqrt{\frac{\beta t}{N}}W_{Nj}u_N v_j + \frac{\beta t}{N}u_Nu_N^*v_jv_j^* - \frac{\beta t}{2N}u_N^2v_j^2,
\end{align*}
and let $\langle \cdot \rangle_t$ be the associated Gibbs average. We let 
\[X(t) = \exp\Big(\imnb s \log \int e^{-H_t(\u,\v)} \rmd \rho(\u,\v)\Big), \]
and
\[\varphi(t) = N\E\left[\left\langle u_N^{(1)}u_N^{*} R^{\vtt}_{1,*}\right\rangle_t X(t)\right].\]
Observe that $\varphi(1)$ is the quantity we seek to analyze. We will use the following error bound on Taylor's expansion:
\[\abs{\varphi(1) - \varphi(0) - \varphi'(0)} \le \sup_{0\le t\le 1} |\varphi''(t)|,\]
to approximate $\varphi(1)$ by $\varphi(0) + \varphi'(0)$. Since $P_{\utt}$ is centered, we have $\varphi(0) = 0$. With a computation similar to the one leading to Lemma~\ref{derivative_phi_planted_lemma}, the time derivative $\varphi'(t)$ is a sum of terms of the form 
\[N\beta \E\left[\left\langle u_N^{(1)}u_N^{*}u_N^{(a)}u_N^{(b)} R^{\vtt}_{1,*}R^{\vtt}_{a,b}\right\rangle_t X(t)\right],\]
for $(a,b) \in \{(1,*),(2,*),(1,2),(2,3)\}$. At $t=0$ all terms vanish expect when $(a,b)=(1,*)$ and we get
\[\varphi'(0) = N\beta \E\left[\left\langle (R^{\vtt}_{1,*})^2\right\rangle_0 X(0)\right].\]
Now we wish to replace the time index $t=0$ in the above quantity by the time index $t=1$. Similarly to $\varphi$, the derivative of the function $t \mapsto N\beta \E[\langle (R^{\vtt}_{1,*})^2\rangle_t X(t)]$, is a sum of terms of the form
\[N\beta^2 \E\left[\left\langle u_N^{(a)}u_N^{(b)} (R^{\vtt}_{1,*})^2R^{\vtt}_{a,b}\right\rangle_t X(t)\right].\]
By boundedness of the $\utt$ variables and H\"older's inequality, this is bounded by
\begin{align*}
N\beta^2  K_{\utt}^4 \E\left[\left\langle \abs{(R^{\vtt}_{1,*})^2R^{\vtt}_{a,b}}\right\rangle_t\right]
&\le N\beta^2 K_{\utt}^4 \E\left[\left\langle |R^{\vtt}_{1,*}|^{3}\right\rangle_t\right]\\
&\le N\beta^2 K_{\utt}^4K \E\left[\left\langle |R^{\vtt}_{1,*}|^{3}\right\rangle\right]\\
&\le \frac{K\beta^2}{\sqrt{N}},
\end{align*}
where the second bound is by Lemma~\ref{bound_time_dependent_average}, and the last bound is a consequence of Proposition~\ref{fourth_moment_planted} (and Jensen's inequality). 
Therefore
\[\abs{\varphi'(0) - N\beta \E\left[\left\langle (R^{\vtt}_{1,*})^2\right\rangle X(1)\right]} \le \frac{K}{\sqrt{N}}.\]
Similarly, we control the second derivative $\varphi''$. This can be written as a finite  sum of terms of the form
\[N\beta^2 \E\left[\left\langle u_N^{(1)}u_N^{*}u_N^{(a)}u_N^{(b)}u_N^{(c)}u_N^{(d)} R^{\vtt}_{1,*}R^{\vtt}_{a,b}R^{\vtt}_{c,d}\right\rangle_t X(t)\right],\]
which are bounded in the same way by 
\[N\beta^2 K_{\utt}^6 \E\left[\left\langle \abs{ R^{\vtt}_{1,*}R^{\vtt}_{a,b}R^{\vtt}_{c,d}}\right\rangle_t \right] \le \beta^2 K_{\utt}^6 \frac{K}{\sqrt{N}}.\]
Therefore $|\varphi''| \le K/\sqrt{N}$. 
We end up with 
\begin{equation}\label{important_1}
N\E\left[\left\langle R^{\utt}_{1,*} R^{\vtt}_{1,*}\right\rangle e^{\imnb s \log L}\right] = N\beta \E\left[\left\langle (R^{\vtt}_{1,*})^2\right\rangle e^{\imnb s \log L}\right] + \delta,
\end{equation}
where $|\delta| \le K/\sqrt{N}$ whenever $(\alpha,\beta)$ satisfy the conditions of  Proposition~\ref{fourth_moment_planted}. 

\subsection{Cavity on $M$}
By symmetry of the $\vtt$ variables,  
\begin{align*} 
&N\E\left[\left\langle (R^{\vtt}_{1,*})^2\right\rangle e^{\imnb s \log L}\right] = M\E\left[\left\langle v_M^{(1)}v_M^{*}R^{\vtt}_{1,*}\right\rangle e^{\imnb s \log L}\right]\\
&~~= \frac{M}{N}\E\left[\left\langle (v_M^{(1)}v_M^{*})^2\right\rangle e^{\imnb s \log L}\right] +M\E\left[\left\langle v_M^{(1)}v_M^{*}R^{\vtt -}_{1,*}\right\rangle e^{\imnb s \log L}\right]. 
\end{align*}
Now we execute the same argument as above with the roles of $\utt$ and $\vtt$ flipped to prove that 
\[\E\left[\left\langle (v_M^{(1)}v_M^{*})^2\right\rangle e^{\imnb s \log L}\right] = \E\left[e^{\imnb s \log L}\right] +\delta, \]
and
\[M\E\left[\left\langle v_M^{(1)}v_M^{*}R^{\vtt -}_{1,*}\right\rangle e^{\imnb s \log L}\right] = M\beta\E\left[\left\langle R^{\utt}_{1,*}R^{\vtt}_{1,*}\right\rangle e^{\imnb s \log L}\right] +\delta, \]
where $|\delta| \le K (M/N^{3/2} \vee 1/\sqrt{N})$. Here we use the interpolating Hamiltonian
\begin{align*}
-H_t(\u,\v) &= \sum_{j=1}^{ M-1} \sum_{i=1}^N \sqrt{\frac{\beta}{N}}W_{ij}u_i v_j + \frac{\beta}{N}u_iu_i^*v_jv_j^*- \frac{\beta}{2N}u_i^2v_j^2 \\
&~~+  \sum_{i=1}^N \sqrt{\frac{\beta t}{N}}W_{iM}u_i v_M + \frac{\beta t}{N}u_iu_i^*v_Mv_M^*- \frac{\beta t}{2N}u_i^2v_M^2,
\end{align*}
and similarly define the random variable $X(t) = \exp\big(\imnb s \log \int e^{-H_t(\u,\v)} \rmd \rho(\u,\v)\big)$.
After executing the argument, we obtain 
\begin{equation}\label{important_2}
N\E\left[\left\langle (R^{\vtt}_{1,*})^2\right\rangle e^{\imnb s \log L}\right]  = \frac{M}{N} \E\left[e^{\imnb s \log L}\right] + M\beta \E\left[\left\langle R^{\utt}_{1,*} R^{\vtt}_{1,*}\right\rangle e^{\imnb s \log L}\right] + \delta.
\end{equation}
From~\eqref{important_1} and~\eqref{important_2}, we obtain 
\[N\E\left[\left\langle R^{\utt}_{1,*} R^{\vtt}_{1,*}\right\rangle e^{\imnb s \log L}\right]  = \frac{M}{N} \beta \E\left[e^{\imnb s \log L}\right]  + M \beta^2 \E\left[\left\langle R^{\utt}_{1,*} R^{\vtt}_{1,*}\right\rangle e^{\imnb s \log L}\right] +\delta, \]
where $|\delta| \le K (M/N^{3/2} \vee 1/\sqrt{N})$. For $M= \alpha N + \bigo(\sqrt{N})$, we arrive at
\[N\E\left[\left\langle R^{\utt}_{1,*} R^{\vtt}_{1,*}\right\rangle e^{\imnb s \log L}\right]  = \frac{\alpha \beta}{1-\alpha \beta^2}\E\left[e^{\imnb s \log L}\right] + \delta,\]
with $|\delta| \le K/\sqrt{N}$, and this finishes the proof.% for the planted model and $l=*$. 

\section{Proof of Proposition~\ref{fourth_moment_planted}} 
\label{sxn:proof_of_fourth_moment_planted}
This section is about overlap convergence in the planted model. As explained in the main text, the proof is in several steps. We first present a proof of convergence of the second moment of the overlaps that does not rely on the cavity method, but on a \emph{quadratic replica coupling} scheme of~\cite{guerra2002quadratic}. Then we present the interpolation argument as a fixed overlap that will allow us to prove the crude convergence bound~\eqref{crude_bound}. Finally we execute a round of the cavity method to prove convergence of the fourth moment.

\subsection{Convergence of the second moment}
\label{sxn:proof_of_convergence_second_moment}
\begin{proposition}\label{second_moment_planted}
For all $\alpha,\beta$ such that $K_{\utt}^4K_{\vtt}^4\alpha\beta^2 < 1$, there exists $K = K(\alpha,\beta) <\infty$ such that
\[\E\left\langle (R^{\utt}_{1,*})^2\right\rangle \vee \E\left\langle (R^{\vtt}_{1,*})^2\right\rangle \le \frac{K}{N^2}.\]
\end{proposition}
Of course, by the Nishimori property, this is also a statement about the overlaps between two independent replicas.   

\begin{proof}
Let $\sigma_{\utt}$ and $\sigma_{\vtt}$ be the sub-Gaussian parameters of $P_{\utt}$ and $P_{\vtt}$ respectively. We since $P_{\utt}$ and $P_{\vtt}$ have unit variance, we have $1 \le \sigma_{\utt}^2 \le K_{\utt}^2$ and similarly for $P_{\vtt}$.

We start with the $\utt$-overlap. Let us define the function
\[\Phi_{\utt}(\lambda) = \frac{1}{N}\E \log \int \exp\left(-H(\u,\v) + \frac{\lambda}{2} N (R_{1,*}^{\utt})^2\right) \rmd \rho(\u,\v).\]
The outer expectation is on $\Y \sim \P_{\beta}$ (or equivalently on $\u^*$, $\v^*$ and $\W$ independently). 
A simple inspection shows that the above function is convex and increasing in $\lambda$, and
\[\Phi_{\utt}'(0) = \half\E\left \langle (R^{\utt}_{1,*})^2\right \rangle.\] 
The convexity then implies for all $\lambda \ge 0$,
\[\frac{\lambda}{2}\E\left \langle (R^{\utt}_{1,*})^2\right \rangle \le \Phi_{\utt}(\lambda) - \Phi_{\utt}(0). \]
Of course $\Phi_{\utt}(0) = \frac{1}{N}\E_{\P_{\beta}} \log L(\Y;\beta) \ge 0$ by Jensen's inequality, so it remains to upper bound $\Phi_{\utt}(\lambda)$.
To this end we consider the interpolation
\[\Phi_{\utt}(\lambda,t) = \frac{1}{N}\E \log \int \exp\left(-H_t(\u,\v) + \frac{\lambda}{2} N (R_{1,*}^{\utt})^2\right) \rmd \rho(\u,\v),\]
where 
\[-H_t(\u,\v) = \sum_{i,j}\sqrt{\frac{\beta t}{N}}W_{ij}u_iv_j + \frac{\beta}{N}u_iu_i^*v_jv_j^*- \frac{\beta t}{2N} u_i^2v_j^2.\]
Notice that the planted (middle) term in the Hamiltonian is left unaltered. The time derivative is
\[\partial_t \Phi_{\utt}(\lambda,t) = -\frac{\beta}{2}\E\left \langle (R^{\utt}_{1,2})^2\right\rangle_{\lambda,t} \le 0,\]
where $\left \langle \cdot \right\rangle_{\lambda,t}$ is the Gibbs average w.r.t\ $-H_t(\u,\v) + \frac{\lambda}{2} N (R_{1,*}^{\utt})^2$. Therefore
\begin{align*}
\Phi_{\utt}(\lambda) \le \Phi_{\utt}(\lambda,0) &= \frac{1}{N}\E \log \int \exp\left(\beta N R_{1,*}^{\utt}R_{1,*}^{\vtt} + \frac{\lambda}{2} N (R_{1,*}^{\utt})^2\right) \rmd \rho(\u,\v)\\
&\le \frac{1}{N}\E \log \int \exp\left(\frac{\alpha\beta^2\sigma_\vtt^2 \widehat{v}  + \lambda}{2} N (R_{1,*}^{\utt})^2\right) \rmd P_{\utt}^{\otimes N}(\u),
\end{align*}
where we have used the sub-Gaussianity of $P_{\vtt}$, and let $\widehat{v} = \frac{1}{M}\sum_{j=1}^Mv_j^{*2}$. (Here, we have abused notation and let $\alpha = \frac{M}{N}$. This will not cause any problems.) 
Next we introduce an independent r.v.\ $g \sim \normal(0,1)$, exchange integrals by Fubini's theorem, and continue:
\begin{align*}
&\frac{1}{N}\E \log \E_g \left[\int \exp\left(\sqrt{(\alpha\beta^2\sigma_\vtt^2 \widehat{v}  + \lambda) N} R_{1,*}^{\utt}g\right) \rmd P_{\utt}^{\otimes N}(\u)\right]\\
&\le\frac{1}{N}\E \log \E_g  \left[\exp\left(\frac{\alpha\beta^2\sigma_\vtt^2 \widehat{v}  + \lambda}{2} \sigma_{\utt}^2 \widehat{u} g^2\right)\right],
\end{align*}
where we use the sub-Gaussianity of $P_{\utt}$, and let $\widehat{u} = \frac{1}{N}\sum_{i=1}^N u_i^{*2}$. We bound $\widehat{u}$ and $\widehat{v}$ by $K_{\utt}^2$ and $K_{\vtt}^2$ respectively and integrate on $g$ to obtain the upper bound
\begin{equation*}
\Phi_{\utt}(\lambda) \le -\frac{1}{2N} \log \left(1 - (\alpha\beta^2\sigma_\vtt^2 K_{\vtt}^2  + \lambda)\sigma_{\utt}^2 K_{\utt}^2\right),
\end{equation*}
valid as long as $(\alpha\beta^2\sigma_\vtt^2 K_{\vtt}^2  + \lambda)\sigma_{\utt}^2 K_{\utt}^2 <1$. Letting $\lambda = (1-\alpha\beta^2\sigma_\vtt^2 K_{\vtt}^2\sigma_{\utt}^2 K_{\utt}^2)/(2\sigma_{\utt}^2 K_{\utt}^2) >0$, we obtain
\[\E\left \langle (R^{\utt}_{1,*})^2\right \rangle \le \frac{K(\alpha,\beta)}{N},\]
with $K(\alpha,\beta) = \frac{2\sigma_{\utt}^2 K_{\utt}^2\log((1-\alpha\beta^2\sigma_\vtt^2 K_{\vtt}^2\sigma_{\utt}^2 K_{\utt}^2)/2)}{(1-\alpha\beta^2\sigma_\vtt^2 K_{\vtt}^2\sigma_{\utt}^2 K_{\utt}^2)}$.

We use the exact same argument for the $\vtt$-overlaps. We define $\Phi_\vtt(\lambda)$ in the same way by replacing the quadratic term $\frac{\lambda}{2}N (R^\utt_{1,*})^2$ by $\frac{\lambda}{2}N(R^\vtt_{1,*})^2$ and obtain
\[\Phi_\vtt(\lambda) \le -\frac{1}{2N} \log \left(1 - (\beta^2\sigma_\utt^2 K_{\utt}^2  + \lambda)\alpha\sigma_{\vtt}^2 K_{\vtt}^2\right).\]
We choose $\lambda = (1-\alpha\beta^2\sigma_\vtt^2 K_{\vtt}^2\sigma_{\utt}^2 K_{\utt}^2)/(2\alpha\sigma_{\vtt}^2 K_{\vtt}^2)$ and use the same convexity argument to obtain
\[\E\left \langle (R^{\vtt}_{1,*})^2\right \rangle \le \frac{K'(\alpha,\beta)}{N},\]
with $K'(\alpha,\beta) = \frac{2\alpha\sigma_{\vtt}^2 K_{\vtt}^2\log((1-\alpha\beta^2\sigma_\vtt^2 K_{\vtt}^2\sigma_{\utt}^2 K_{\utt}^2)/2)}{(1-\alpha\beta^2\sigma_\vtt^2 K_{\vtt}^2\sigma_{\utt}^2 K_{\utt}^2)}$.
\end{proof}

\subsection{Interpolation bound at fixed overlap}
\label{sxn:interpolation_at_fixed_overlap}
In this section we present and prove an interpolation bound on the free energy of a subpopulation of configurations having a fixed overlap with the planted spike $(\u^*,\v^*)$. This is a key step in proving the crude bound~\eqref{crude_bound}. 
\begin{proposition}\label{interpolation_bound_fixed_overlap}
Fix $\u^* \in \R^{N}, \v^* \in \R^{M}$ with $\twonorm{\u^*}^2/N \le K_{\utt}^2$ and $\twonorm{\v^*}^2/M \le K_{\vtt}^2$. Let $\alpha = \frac{M}{N}$ and $\Delta = \alpha \beta^2 \sigma_\utt^2\sigma_\vtt^2K_\utt^2K_\vtt^2-1$. For $m \in \R\setminus \{0\}$, $\epsilon\ge 0$, let $A_{\utt}$ be the event
\begin{align*}
A_{\utt} = 
\begin{cases}
R^{\utt}_{1,*} \in [m,m+\epsilon) & \mbox{if } m >0, \\
R^{\utt}_{1,*} \in (m-\epsilon,m] & \mbox{if } m <0.
\end{cases}
\end{align*}
Define $A_{\vtt}$ similarly. We have 
\begin{equation}\label{interpolation_bound_u}
\frac{1}{N}\E\log \int \indi (A_{\utt}) e^{-H(\u,\v)}\rmd \rho(\u,\v) \le \frac{\Delta}{2\sigma_\utt^2K_\utt^2}  m^2 + \alpha \beta K_{\vtt}^2\epsilon,
\end{equation}
and
\begin{equation}\label{interpolation_bound_v}
\frac{1}{N}\E\log \int \indi(A_{\vtt}) e^{-H(\u,\v)}\rmd \rho(\u,\v) \le \frac{\Delta}{2\alpha \sigma_\vtt^2K_\vtt^2} m^2 + \beta K_{\utt}^2\epsilon.
\end{equation}
The expectation $\E$ is over the Gaussian disorder $\W$.
\end{proposition}
\begin{proof}
We only prove~\eqref{interpolation_bound_u}. The bound~\eqref{interpolation_bound_v} follows by flipping the roles of $\u$ and $\v$.  
We consider the interpolating Hamiltonian
\begin{align*}
-H_t(\u,\v) = \sum_{i,j}\sqrt{\frac{\beta t}{N}}W_{ij}u_iv_j + \frac{\beta t}{N}u_iu_i^*v_jv_j^*- \frac{\beta t}{2N} u_i^2v_j^2
+ \sum_{j=1}^M (1-t)\beta mv_jv_j^*,
\end{align*}
and let
\[\varphi(t) = \frac{1}{N}\E\log \int \indi\{R^{\utt}_{1,*} \in [m,m+\epsilon)\} e^{-H_t(\u,\v)}\rmd \rho(\u,\v).\]
We have
\[\varphi'(t) = -\frac{\beta}{2}\E \left\langle  R^{\utt}_{1,2}R^{\vtt}_{1,2} \right\rangle_t + \beta\E \left\langle  R^{\utt}_{1,*}R^{\vtt}_{1,*} \right\rangle_t - \beta m\E \left\langle  R^{\vtt}_{1,*} \right\rangle_t.\]
The first term in the above expression is $\le 0$, and since the overlap $R_{1,*}^\utt$ is constrained to be close to $m$ we have $\abs{\E \left\langle  (R^{\utt}_{1,*}-m)R^{\vtt}_{1,*} \right\rangle_t} \le  \alpha K_{\vtt}^2\epsilon$. So $\varphi'(t) \le \alpha K_{\vtt}^2\epsilon$. Moreover, the variables $\u$ and $\v$ decouple at $t=0$ and one can write
\[\varphi(1) \le  \frac{1}{N}\log \Pr\left(A_{\utt}\right) + \frac{1}{N}\sum_{j=1}^M\log \E_{v}\left[ e^{\beta mvv_j^*}\right] + K_{\vtt}^2\epsilon.\] 
By sub-Gaussianity of the prior $P_{\vtt}$ we have $\E_{v}\left[ e^{\beta mvv_j^*}\right] \le e^{\beta^2 \sigma_{\vtt}^2m^2 v_j^{*2}/2}$. On the other hand, for a fixed parameter $\gamma$ of the same sign as $m $, we have
\[ \frac{1}{N}\log \Pr \left(A_{\utt}\right) \le -\gamma m+\frac{1}{N}\sum_{i=1}^N\log \E_{u}[e^{\gamma uu_i^*}] \le -\gamma m + \frac{1}{2N}\sum_{i=1}^N u_i^{*2}\sigma_{\utt}^2\gamma^2.\] 
The last inequality uses sub-Gaussianity of $P_{\utt}$. We minimize this quadratic w.r.t\ $\gamma$ and obtain
\[\varphi(1) \le - \frac{m^2}{2\sigma_{\utt}^2\widehat{u}} + \frac{M}{2N} \beta^2 \sigma_{\vtt}^2\widehat{v} m^2 + \alpha K_{\vtt}^2\epsilon,\]
where $\widehat{u} = \frac{1}{N}\sum_{i=1}^Nu_i^{*2}$ and $\widehat{v}=\frac{1}{M}\sum_{j=1}^Mv_j^{*2}$. We upper bound the latter two numbers by $K_{\utt}^2$ and $K_{v}^2$ respectively.
\end{proof}

\subsection{Overlap concentration (proof of~\eqref{crude_bound})}
\label{sxn:proof_of_crude_bound}
Here we prove convergence of the overlaps to zero in probability. We first state a useful and standard result of concentration of measure.  

\begin{lemma}\label{gaussian_concentration}
Let $\Y = \sqrt{\frac{\beta}{N}}\u^*\v^{*\top} + \W$, where the planted vectors $\u^*$ and $\v^*$ are fixed, and $W_{ij} \sim \normal(0,1)$. For a Borel set $A \subset \R^{M+N}$, let  
\[Z = \int_A e^{-H(\u,\v)}\rmd \rho(\u,\v).\]
We have for every $t \ge 0$,
\[\Pr\left(\abs{\log Z - \E\log Z} \ge Nt\right) \le 2 e^{-\frac{Nt^2}{2\beta K_{\utt}^2K_{\vtt}^2}}.\]
(Here $\Pr$ and $\E$ are conditional on $\u^*$ and $\v^*$.)
\end{lemma}
\begin{proof}
We simply observe that the function $\W \mapsto \log Z$ is Lipschitz with constant $\sqrt{N\beta \alpha K_{\utt}^2K_{\vtt}^2}$. The result follows from concentration of Lipschitz functions of Gaussian r.v.'s \citep[this is the Borell-Tsirelson-Ibragimov-Sudakov inequality; see][Theorem 5.6]{boucheron2013concentration}.
\end{proof}

\begin{proposition}\label{convergence_in_probability}
Let $\alpha,\beta$ such that $\alpha \beta^2 \sigma_\utt^2\sigma_\vtt^2K_\utt^2K_\vtt^2<1$, and $\epsilon>0$. There exist constants $c = c(\epsilon,\alpha,\beta,K_\utt,K_\vtt)>0$ and $K = K(K_{\utt},K_{\vtt})>0$ such that
\[\E \left\langle \indi \{ |R^{\utt}_{1,*}| \ge \epsilon \} \right\rangle  \vee \E \left\langle \indi \{ |R^{\vtt}_{1,*}| \ge \epsilon \} \right\rangle \le \frac{K}{\epsilon^2} e^{-cN}.\] 
\end{proposition}
\begin{proof}
We only prove the assertion for the $\utt$-overlap since the argument is strictly the same for the $\vtt$-overlap.
 
For $\epsilon,\epsilon'>0$, we can write the decomposition
\begin{align*}
\E \left\langle \indi \{ \abs{R^{\utt}_{1,*}} \ge \epsilon' \} \right\rangle &= \sum_{l \ge 0} \E \left\langle \indi \{ R^{\utt}_{1,*} -\epsilon' \in [l\epsilon,(l+1)\epsilon) \} \right\rangle \\
&~ +\sum_{l \ge 0} \E \left\langle \indi \{ -R^{\utt}_{1,*} +\epsilon' \in [l\epsilon,(l+1)\epsilon) \} \right\rangle,
\end{align*}
where the integer index $l$ ranges over a finite set of size $\le K/\epsilon$. We only treat the generic term in the first sum; the second sum can be handled similarly.  
Fix $m >0, \epsilon>0$. We have
\begin{equation}\label{fraction_1}
\E \left\langle \indi \{ R^{\utt}_{1,*} \in [m,m+\epsilon) \} \right\rangle = \E \left[\frac{\int \indi \{ R^{\utt}_{1,*} \in [m,m+\epsilon)\}e^{-H(\u,\v)}\rmd\rho(\u,\v)}{\int e^{-H(\u,\v)}\rmd\rho(\u,\v)}\right].
\end{equation}
Let 
\[A = \frac{1}{N} \E_{\W}\log \int \indi\{R^{\utt}_{1,*} \in [m,m+\epsilon)\}e^{-H(\u,\v)}\rmd\rho(\u,\v),\]
and
\[B = \frac{1}{N} \E_{\W}\log \int e^{-H(\u,\v)}\rmd\rho(\u,\v).\]
By concentration over the Gaussian disorder, Lemma~\ref{gaussian_concentration}, for any $u \ge0$, we simultaneously have
\[\frac{1}{N}\log \int \indi\{R^{\utt}_{1,*} \in [m,m+\epsilon)\}e^{-H(\u,\v)}\rmd\rho(\u,\v) - A \le u,\] 
and
\[\frac{1}{N} \log \int e^{-H(\u,\v)}\rmd\rho(\u,\v) - B \ge -u,\]
with probability at least $1-4e^{-Nu^2/(2\beta K_{\utt}^2 K_{\vtt}^2)}$. 
On the complement event we simply upper bound the fraction~\eqref{fraction_1} by 1. Therefore, we have
\begin{equation*}\label{fraction_2}
\E \left\langle \indi \{ R^{\utt}_{1,*} \in [m,m+\epsilon) \} \right\rangle \le \E_{\u^*,\v^*}\left[e^{N( A- B+2u)}\right] + 4e^{-Nu^2/(2\beta K_{\utt}^2 K_{\vtt}^2)}.
\end{equation*}
By Proposition~\ref{interpolation_bound_fixed_overlap} we have $A \le \frac{\Delta}{2\sigma_\utt^2K_\utt^2}  m^2 + \alpha\beta K_{\vtt}^2\epsilon$ deterministically over $\u^*$ and $\v^*$. Now it remains to control $\E_{\u^*,\v^*}\left[e^{-NB}\right]$.
\begin{lemma}\label{upperbound_B}
We have $\E_{\u^*,\v^*}\left[e^{-NB}\right] \le 2 e^{-N\E_{\u^*,\v^*}[B]}$.
\end{lemma}
Moreover, observe that
\begin{align*}
\E_{\u^*,\v^*}[B] &= \frac{1}{N} \E\log \int e^{-H(\u,\v)}\rmd\rho(\u,\v) \\
&=  \frac{1}{N} \E_{\P_{\beta}} \log L(\Y;\beta) \\
&= \frac{1}{N} \E_{\P_{0}} L(\Y;\beta)\log L(\Y;\beta) \ge 0.
\end{align*} 
Positivity is obtained by Jensen's inequality and convexity of $x \mapsto x\log x$. In view of the above, Lemma~\ref{upperbound_B} means that the random variable $B$ is ``essentially" positive. Therefore,
\[\E \left\langle \indi \{ R^{\utt}_{1,*} \in [m,m+\epsilon) \} \right\rangle \le 2e^{N(\delta+2u)} + 4e^{-Nu^2/(2\beta K_{\utt}^2 K_{\vtt}^2)},\]
where $\delta = \frac{\Delta}{2\sigma_\utt^2K_\utt^2}  m^2 + \alpha\beta K_{\vtt}^2\epsilon$.
We let $u = -\delta/3 \ge 0$, and $m = \epsilon'+l\epsilon$. Since $\Delta <0$, $\Delta m^2 \le \Delta \epsilon'^2$. 
Now we let $\epsilon = - \frac{\Delta}{4\alpha\beta\sigma_\utt^2K_\utt^2K_\vtt^2}  \epsilon'^2$ so that $\delta \le \frac{3\Delta}{4\sigma_\utt^2K_\utt^2} \epsilon'^2 <0$.
 \end{proof}

\vspace{.3cm}
\noindent\begin{proofof}{Lemma~\ref{upperbound_B}}
We abbreviate $\E_{\u^*,\v^*}$ by $\E$. We have
\[\E\left[e^{N(\E [B] - B)}\right] = \int_{-\infty}^{+\infty} e^t \Pr\left(N(\E [B] - B) \ge t\right) \rmd t \le 1+ \int_0^{+\infty} e^t \Pr\left(N(\E [B] - B) \ge t\right)\rmd t. \]
Now we bound the lower tail probability. The r.v.\ $B$, seen as a function of the vector $[\u^*|\v^*] \in \R^{N+M}$ is jointly convex (the Hessian can be easily shown to be positive semi-definite), and Lipschitz with constant $\beta K_\utt K_\vtt\sqrt{\frac{\alpha K_\utt^2+\alpha^2K_\vtt^2}{N}}$ with respect to the $\ell_2$ norm. Under the above conditions, a bound on the lower tail of deviation of $B$  is available; this is (one side of) Talagrand's inequality~\citep[see][Theorem 7.12]{boucheron2013concentration}. Therefore, we have for all $t\ge 0$
\[\Pr\left( B - \E [B]\le -t\right)\le e^{-Nt^2/2K^2},\] 
where $K^2 = \alpha\beta^2K_\utt^2 K_\vtt^2( K_\utt^2+\alpha K_\vtt^2)$. Thus,
\begin{align*}
\E\left[e^{N(\E [B] - B)}\right] &\le 1+ \int_0^{+\infty} e^t e^{-t^2/(2NK^2)}\rmd t\\
&= 1+ K\sqrt{N}e^{NK^2/2} \int_{K\sqrt{N}}^{+\infty} e^{-t^2/2}\rmd t\\
&\le 2. 
\end{align*}
The last inequality is a restatement of the fact $\Pr( g \ge t) \le \frac{e^{-t^2/2}}{\sqrt{2\pi}t}$ where $g \sim \normal(0,1)$.
\end{proofof}

\subsection{Convergence of the fourth moment}
In this section we prove that for all $\alpha,\beta$ such that $\alpha \beta^2 \sigma_\utt^2\sigma_\vtt^2K_\utt^2K_\vtt^2<1$, we have
\[\E\left \langle (R^{\utt}_{1,2})^4\right \rangle \vee \E\left \langle (R^{\vtt}_{1,2})^4\right \rangle \le \frac{K(\alpha,\beta)}{N^2}.\]
We proceed as follows. Let
\[M = \max\left\{\E\left \langle (R^{\utt}_{1,2})^4\right \rangle, \E\left \langle (R^{\vtt}_{1,2})^4\right \rangle \right\}.\]
We prove that for $\epsilon>0$, the following self-boundedness properties hold:
\begin{align}
\E\left \langle (R^{\utt}_{1,2})^4\right \rangle &\le \alpha \beta^2\E\left \langle (R^{\utt}_{1,2})^4\right \rangle + K\epsilon M + \delta \label{cavity_bound_fourth_moment_u},\\
\E\left \langle (R^{\vtt}_{1,2})^4\right \rangle &\le \alpha \beta^2\E\left \langle (R^{\vtt}_{1,2})^4\right \rangle + K\epsilon M + \delta,
\label{cavity_bound_fourth_moment_v}
\end{align} 
where $\delta \le K/N^2 + K/\epsilon^2 e^{-c(\epsilon)N}$. This implies the desired result by letting $\epsilon$ be sufficiently small (e.g., $\epsilon = (1-\alpha\beta^2)/2$). We prove~\eqref{cavity_bound_fourth_moment_u} and~\eqref{cavity_bound_fourth_moment_v} using the cavity method, i.e.\ by isolating the effect of the last variables $u_N$ and $v_M$, one at a time. We prove~\eqref{cavity_bound_fourth_moment_u} in full detail, then briefly highlight how~\eqref{cavity_bound_fourth_moment_v} is obtained in a similar way.        

By symmetry between the $\utt$ variables, we have
\begin{align*}
\E\left \langle (R^{\utt}_{1,*})^4\right \rangle &= \E\left \langle u^{(1)}_{N}u^{*}_{N}(R^{\utt}_{1,*})^3\right \rangle \\
&= \E\left \langle u^{(1)}_{N}u^{*}_{N}\Big(R^{\utt - }_{1,*} + \frac{1}{N}u^{(1)}_{N}u^{*}_{N}\Big)^3\right \rangle. 
\end{align*}
Expanding the term $\big(R^{\utt - }_{1,*} + \frac{1}{N}u^{(1)}_{N}u^{*}_{N}\big)^3$ we obtain
\begin{equation}\label{first_upper_bound_fourth_moment}
\E\left \langle (R^{\utt}_{1,*})^4\right \rangle \le \E\left \langle u^{(1)}_{N}u^{*}_{N}(R^{\utt - }_{1,*})^3\right \rangle + \frac{K_\utt^4}{N}\E\left \langle (R^{\utt - }_{1,*})^2 \right \rangle + \frac{K_\utt^6}{N^2}\E\left \langle \big|R^{\utt - }_{1,*}\big| \right \rangle + \frac{K_\utt^8}{N^3}.
\end{equation}
We have already proved convergence of the second moment (Proposition~\ref{second_moment_planted}), hence $\E \langle (R^{\utt - }_{1,*})^2  \rangle \le K/N$ and $\E \langle |R^{\utt - }_{1,*}|  \rangle \le K/\sqrt{N}$. Now we need to control the leading term involving $(R^{\utt - }_{1,*})^3$. The next proposition shows that this quantity can be related back to $(R^{\utt}_{1,*})^4$, plus additional higher-order terms. This is is achieved through the cavity method.  
\begin{proposition}\label{self_consistent_fourth_moment_u} 
For $\alpha,\beta \ge 0$, there exists a constant $K = K(\alpha,\beta,K_{\utt},K_{\vtt}) >0$ such that
\begin{equation}\label{cavity_u_1}
\E\left \langle u^{(1)}_{N}u^{*}_{N}(R^{\utt-}_{1,*})^3\right \rangle = \beta\E\left \langle(R^{\utt}_{1,*})^3R^{\vtt}_{1,*}\right \rangle + \delta_1,
\end{equation}
where 
\[|\delta_1| \le K\sum_{a,b,c,d} \E\left \langle \abs{(R^{\utt-}_{1,*})^3R^{\vtt}_{a,b}R^{\vtt}_{c,d}}\right \rangle.\]
Moreover, 
\begin{equation}\label{cavity_v_1}
\E\left \langle(R^{\utt}_{1,*})^3R^{\vtt}_{1,*}\right \rangle  = \alpha\beta\E\left \langle(R^{\utt}_{1,*})^4\right \rangle + \delta_2,
\end{equation}
where
\[|\delta_2| \le K\sum_{a,b,c,d} \E\left \langle \abs{(R^{\utt}_{1,*})^3R^{\utt}_{a,b}R^{\utt}_{c,d}}\right \rangle.\]
\end{proposition}
From Proposition~\ref{self_consistent_fourth_moment_u} we deduce 
\[\E\left \langle u^{(1)}_{N}u^{*}_{N}(R^{\utt-}_{1,*})^3\right \rangle =\alpha\beta^2\E\left \langle(R^{\utt}_{1,*})^4\right \rangle + \delta,\]
where $\delta = \delta_1 + \delta_2$. Plugging into~\eqref{first_upper_bound_fourth_moment}, we obtain
\[\E\left \langle (R^{\utt}_{1,*})^4\right \rangle \le \alpha\beta^2\E\left \langle(R^{\utt}_{1,*})^4\right \rangle + \frac{K}{N^2} + \delta.\] 
Now we need to control the error term $\delta$, which involves monomials of degree $5$ in the overlaps $R^\utt$ and $R^\vtt$. This is where the a priori bound on the convergence of the overlaps, Proposition~\ref{convergence_in_probability}, is useful. Since the overlaps are bounded, we can write for any $\epsilon >0$,
\begin{align*}
\E\left \langle \abs{(R^{\utt}_{1,*})^3R^{\vtt}_{a,b}R^{\vtt}_{c,d}}\right \rangle &\le \epsilon \E\left \langle \abs{(R^{\utt}_{1,*})^3R^{\vtt}_{a,b}}\right \rangle + K_\utt^6K_\vtt^4 \E\left \langle \indi\{\abs{R^{\vtt}_{c,d}} \ge \epsilon\} \right \rangle\\
&= \epsilon \E\left \langle \abs{(R^{\utt}_{1,*})^3R^{\vtt}_{a,b}}\right \rangle + K_\utt^6K_\vtt^4 \E\left \langle \indi\{\abs{R^{\vtt}_{1,*}} \ge \epsilon\} \right \rangle,
\end{align*}
where the last line is a consequence of the Nishimori property. Now we use H\"older's inequality on the first term:
\begin{align*}
\E\left \langle \abs{(R^{\utt}_{1,*})^3R^{\vtt}_{a,b}}\right \rangle &\le \left(\E\left \langle \abs{(R^{\utt}_{1,*})^4}\right \rangle\right)^{3/4} \left(\E\left \langle \abs{(R^{\vtt}_{a,b}})^4\right \rangle\right)^{1/4} \\
&=\left(\E\left \langle \abs{(R^{\utt}_{1,*})^4}\right \rangle\right)^{3/4} \left(\E\left \langle \abs{(R^{\vtt}_{1,*}})^4\right \rangle\right)^{1/4}\\
&\le M.
\end{align*}
Using Proposition~\ref{convergence_in_probability}, we have $\E \langle \indi\{|R^{\vtt}_{1,*}| \ge \epsilon\}  \rangle \le K e^{-cN}/\epsilon^2$. Therefore,
\[|\delta_1| \le K\epsilon M + \frac{K}{\epsilon^2} e^{-cN}.\]
It is clear that we can use the same argument to bound $\delta_2$, so we end up with
\[\E\left \langle (R^{\utt}_{1,*})^4\right \rangle \le \alpha\beta^2\E\left \langle(R^{\utt}_{1,*})^4\right \rangle + \frac{K}{N^2} +  K\epsilon M + \frac{K}{\epsilon^2} e^{-cN},\] 
thereby proving~\eqref{cavity_bound_fourth_moment_u}. To prove~\eqref{cavity_bound_fourth_moment_v} we use the same approach. We write
\begin{align*}
\E\left \langle (R^{\vtt}_{1,*})^4\right \rangle &= \frac{M}{N}\E\left \langle v^{(1)}_{M}v^{*}_{M}(R^{\vtt}_{1,*})^3\right \rangle\\
&= \alpha\E\left \langle v^{(1)}_{M}v^{*}_{M}\Big(R^{\vtt-}_{1,*} + \frac{1}{N}v^{(1)}_{M}v^{*}_{M}\Big)^3\right \rangle.
\end{align*}
Then use an equivalent of Proposition~\ref{self_consistent_fourth_moment_u} in this case, which is obtained by flipping the role of the $\utt$ and $\vtt$ variables: 
\[\E\left \langle v^{(1)}_{N}v^{*}_{N}(R^{\vtt-}_{1,*})^3\right \rangle = \beta\E\left \langle(R^{\vtt}_{1,*})^3R^{\utt}_{1,*}\right \rangle + \delta_1,\] 
and
\[\E\left \langle (R^{\vtt}_{1,*})^3R^{\utt}_{1,*}\right \rangle = \beta\E\left \langle(R^{\vtt}_{1,*})^4\right \rangle + \delta_2,\]
where $\delta_1$ and $\delta_2$ are similarly bounded by expectations of monomials of degree $5$ in the overlaps $R^\utt$ and $R^\vtt$. These two quantities are then bounded in exactly the same way.   

\vspace{.5cm}
\noindent\begin{proofof}{Proposition~\ref{self_consistent_fourth_moment_u}}
The proof uses two interpolations; the first one decouples the variable $u_N$ from the rest of the system and allows to obtain~\eqref{cavity_u_1}, and the second one decouples the variable $v_M$ and allows to obtain~\eqref{cavity_v_1}. We start with the former. 

\vspace{.3cm}
\noindent\textbf{Proof of~\eqref{cavity_u_1}.} Consider the interpolating Hamiltonian 
\begin{align*}
-H_t(\u,\v) &=  \sum_{i=1}^{N-1} \sum_{j=1}^{M} \sqrt{\frac{\beta}{N}}W_{ij}u_i v_j + \frac{\beta}{N}u_iu_i^*v_jv_j^*- \frac{\beta}{2N}u_i^2v_j^2 \\
&~~+  \sum_{j=1}^M \sqrt{\frac{\beta t}{N}}W_{Nj}u_N v_j + \frac{\beta t}{N}u_Nu_N^*v_jv_j^* - \frac{\beta t}{2N}u_N^2v_j^2,
\end{align*}
and let $\langle \cdot \rangle_t$ be the associated Gibbs average and $\nu_t(\cdot) = \E\langle \cdot \rangle_t$. The idea is to approximate $\nu_1(f)$ where $f \equiv u^{(1)}_{N}u^{*}_{N}(R^{\utt-}_{1,*})^3$ by $\nu_0(f) + \nu_0'(f)$. Of course one then has to control the second derivative, as dictated by Taylor's approximation
\begin{equation}\label{taylor_expansion_second_order}
\abs{\nu_1(f) - \nu_0(f) - \nu_0'(f)} \le \sup_{0 \le t\le 1} \abs{\nu_t''(f)}.
\end{equation}

We see that at time $t=0$, the variables $u_N$ and $u_N^*$ decouple the Hamiltonian, so  
\begin{equation}\label{value_of_average_at_0}
\nu_0(u^{(1)}_{N}u^{*}_{N}(R^{\utt-}_{1,*})^3) = \E[u_N]\E[u^*_N]  \nu_0((R^{\utt-}_{1,*})^3) = 0.
\end{equation}
On the other hand, by applying Lemma~\ref{derivative_gibbs_average} with $n=1$, we see that $\nu_0'(u^{(1)}_{N}u^{*}_{N}(R^{\utt-}_{1,*})^3)$ is a sum of a few terms of the form
\[\nu_0\big(u^{(1)}_{N}u^{*}_{N}u^{(a)}_{N}u^{(b)}_{N}(R^{\utt-}_{1,*})^3R^{\vtt}_{a,b}\big).\]
Since $P_{\utt}$ has zero mean, all terms in which a variable $u_N^{(a)}$ (for any $a$) appears with degree 1 vanish. We are thus left with one term where $a=1,b=*$, and we get 
\begin{equation}\label{value_of_derivative_at_0}
\nu_0'(u^{(1)}_{N}u^{*}_{N}(R^{\utt-}_{1,*})^3) = \beta\E[(u_N^{(1)})^2]\E[(u_N^{*})^2]  \nu_0((R^{\utt-}_{1,*})^3R^{\vtt}_{1,*}) = \beta\nu_0((R^{\utt-}_{1,*})^3R^{\vtt}_{1,*}).
\end{equation}
Moreover, we see that $\nu_0((R^{\utt-}_{1,*})^3R^{\vtt}_{1,*}) = \nu_0((R^{\utt}_{1,*})^3R^{\vtt}_{1,*})$ since the last variable $u_N$ has no contribution under $\nu_0$. Now we are tempted to replace the average at time $t=0$ by an average at time $t=1$ in the last quantity.
We use Lemmas~\ref{derivative_gibbs_average} and~\ref{bound_time_dependent_average} to justify this. Indeed these lemmas and boundedness of the variables $u_N$ imply
\begin{equation}\label{replace_time_0_to_1}
\abs{\nu_0((R^{\utt}_{1,*})^3R^{\vtt}_{1,*}) - \nu_1((R^{\utt}_{1,*})^3R^{\vtt}_{1,*})} \le  K(\alpha,\beta)\sum_{a,b}\nu(\abs{(R^{\utt}_{1,*})^3R^{\vtt}_{1,*}R^{\vtt}_{a,b})}),
\end{equation}
where $(a,b) \in \{(1,2),(1,*),(2,*),(2,3)\}$. Now we control the second derivative $\sup_t\nu_t''(\cdot)$. In view of Lemma~\ref{derivative_gibbs_average}, we see that taking two derivative of $\nu_t(u^{(1)}_{N}u^{*}_{N}(R^{\utt-}_{1,*})^3)$ creates terms of the form
\[\nu_t\left(u^{(1)}_{N}u^{*}_{N}u^{(a)}_{N}u^{(b)}_{N}u^{(c)}_{N}u^{(d)}_{N}(R^{\utt-}_{1,*})^3R^{\vtt}_{a,b}R^{\vtt}_{c,d}\right),\]
with a larger (but finite) set of combinations $(a,b,c,d)$.
We use Lemma~\ref{bound_time_dependent_average} to replace $\nu_t$ by $\nu_1$ and use boundedness of variables $u_N$ to obtain the bound
\begin{equation}\label{bound_second_derivative}
\abs{\sup_{0 \le t\le 1} \nu_t''\Big(u^{(1)}_{N}u^{*}_{N}(R^{\utt-}_{1,*})^3)\Big)} \le K(\alpha,\beta)\sum_{a,b,c,d} \nu\Big(\Big|(R^{\utt-}_{1,*})^3R^{\vtt}_{a,b}R^{\vtt}_{c,d}\Big|\Big).
\end{equation}
Now putting the bounds and estimates~\eqref{taylor_expansion_second_order}, \eqref{value_of_average_at_0}, \eqref{value_of_derivative_at_0}, \eqref{replace_time_0_to_1}, and \eqref{bound_second_derivative}, we obtain the desired bound~\eqref{cavity_u_1}:
\[\abs{\nu\big(u^{(1)}_{N}u^{*}_{N}(R^{\utt-}_{1,*})^3\big) - \beta \nu((R^{\utt}_{1,*})^3R^{\vtt}_{1,*})} \le K(\alpha,\beta)\sum_{a,b,c,d} \nu\Big(\Big|(R^{\utt-}_{1,*})^3R^{\vtt}_{a,b}R^{\vtt}_{c,d}\Big|\Big).\] 

%\vspace{.3cm}
\noindent\textbf{Proof of~\eqref{cavity_v_1}.} By symmetry of the $\vtt$ variables we have
\[\E\left\langle(R^{\utt}_{1,*})^3R^{\vtt}_{1,*}\right\rangle = \frac{M}{N}\E\left\langle(R^{\utt}_{1,*})^3v_M^{(1)}v_M^{*}\right\rangle.\]
Now we apply the same machinery. Consider the interpolating Hamiltonian 
\begin{align*}
-H_t(\u,\v) &= \sum_{j=1}^{ M-1} \sum_{i=1}^N \sqrt{\frac{\beta}{N}}W_{ij}u_i v_j + \frac{\beta}{N}u_iu_i^*v_jv_j^*- \frac{\beta}{2N}u_i^2v_j^2 \\
&~~+  \sum_{i=1}^N \sqrt{\frac{\beta t}{N}}W_{iM}u_i v_M + \frac{\beta t}{N}u_iu_i^*v_Mv_M^*- \frac{\beta t}{2N}u_i^2v_M^2,
\end{align*}
and let $\langle \cdot \rangle_t$ be the associated Gibbs average and $\nu_t(\cdot) = \E\langle \cdot \rangle_t$. The exact same argument goes through with the roles of $\utt$ and $\vtt$ flipped. For instance, when one takes time derivatives, terms of the form $v_M^{(a)}v_M^{(b)}R^{\utt}_{a,b}$ arise from the Hamiltonian, and one sees that 
\[\nu_0'\big((R^{\utt}_{1,*})^3v_M^{(1)}v_M^{*}\big) = \beta\nu_0\big((R^{\utt}_{1,*})^4\big).\]
 Thus we similarly obtain
\[\abs{\nu\big((R^{\utt}_{1,*})^3v_M^{(1)}v_M^{*}\big) - \beta \nu\big((R^{\utt}_{1,*})^4\big)} \le K(\alpha,\beta)\sum_{a,b,c,d} \nu\Big(\Big|(R^{\utt}_{1,*})^3R^{\utt}_{a,b}R^{\utt}_{c,d}\Big|\Big).\]  
\end{proofof}

\end{document}